\documentclass[preprint ,11pt,french,english]{article}
\usepackage{amsmath,amssymb,amsthm} 
\usepackage{mathrsfs,bm, bbm, enumerate}
\usepackage{graphicx,color,tikz}
\usepackage{caption,subcaption}
\usepackage{babel,geometry}
\usepackage[T1]{fontenc} 
\usepackage[latin9]{inputenc}
\newtheorem{lemma}{\textbf{Lemma}}
\newtheorem{corollary}{\textbf{Corollary}}
\newtheorem{proposition}{\textbf{Proposition}}
\newtheorem{theorem}{\textbf{Theorem}}

\theoremstyle{definition}
\newtheorem{definition}{\textbf{Definition}}

\providecommand{\abs}[1]{\left\lvert#1\right\rvert}

\def\CF{{\widehat{\mathscr{P}}}}
\def\One{\mathbbm{1}} 
\def\D{{\mathcal{D}}}
\def\S{{\mathcal{S}}}

\def\R{{\mathcal{R}}}
\def\Lop{\mathrm{L}} 
\def\C{ \mathbb{C}}
\def\Z{ \mathbb{Z}}

\def\N{ \mathbb{N}}
\def\R{ \mathbb{R}}
\def\Rstar{\mathbb{R} \backslash \{0\}}

\def\drm{\mathrm{d}}
\def\Rstar{ \mathbb{R} \backslash \{0\}}

\def\Der{\mathrm{D}}

\def\CFjump{\widehat{\mathcal{P}}_{\mathrm{jump}}}
\def\Pjump{\mathcal{P}_{\mathrm{jump}}}

\begin{document}


\title{Multidimensional L\'evy White Noises in Weighted Besov Spaces}

\maketitle
\begin{center}
\author{Julien Fageot, Alireza Fallah, and Michael Unser}
 \end{center}


\begin{abstract}
In this paper, we study the Besov regularity of $d$-dimensional L\'evy white noises. More precisely, we describe new sample paths properties of a given white noise in terms of weighted Besov spaces. 
In particular, the smoothness and integrability properties of L\'evy white noises are characterized using the Blumenthal-Getoor indices.
Our techniques rely on wavelet methods and generalized moments estimates for L\'evy white noises.
\end{abstract}


\section{Introduction} \label{sec:intro}

 This paper is dedicated to the study of the regularity of a general $d$-dimensional L\'evy white noise in terms of Besov spaces and is a continuation of our previous work \cite{Fageot2015besov}.
A random process is traditionally defined as a collection $(X_t)$ of random variables  indexed by $t \in \R$, with some adequate properties. For instance, L\'evy processes are described as stochastically continuous random processes with independent and stationary increments  \cite{Applebaum2009levy,Sato1994levy}.
However, it is not possible to define L\'evy white noises in the traditional framework. 
Typically, in the 1D setting, it is tempting to introduce L\'evy white noises as the derivatives of L\'evy processes.
The well-known issue is that L\'evy processes are (almost surely) not differentiable in the classical sense: The derivative of a non-trivial L\'evy process does not have a pointwise interpretation.
An alternative way of introducing random processes is based on the abstract theory of measures on function spaces, as developed by Bogachev \cite{Bogachev2007measure} among others. In this context, a random process is a random variable that takes values in a function space endowed with the adequate measurable structure.
In this spirit, Gelfand \cite{Gelfand1955generalized} and It\^o \cite{Ito1954distributions} have independently introduced the concept of generalized random processes, defined as random elements in the  Schwartz space of generalized functions \cite{Schwartz1966distributions}. This approach was more extensively exposed in \cite[Chapter 3]{GelVil4} and \cite{Ito1984foundations}. The $d$-dimensional Schwartz space has the advantage of being stable by (weak) differentiation, and does not only include not only the $d$-dimensional L\'evy white noises, but also all their derivatives.

\paragraph{Measuring   regularity with Besov spaces}
Since we are considering processes that have no pointwise interpretation, we should consider function spaces with negative smoothness.
When talking about the regularity of random processes, the Sobolev or the H\"older regularities are natural concepts that comes into mind.
In order to be more general, we will question the Besov regularity of L\'evy white noises. 
Besov spaces include both Sobolev and H\"older spaces, and provide a finer measure of the regularity of a function \cite{Triebel2008function,Triebel2010theory}.
Evaluating the Besov regularity of random processes over $\R^d$ requires the  introduction of weights, since they are generally not decreasing towards infinity. Thereafter, we therefore consider weighted Besov spaces or local Besov spaces.

\paragraph{Regularity of L\'evy white noises and related processes}
To the best of our knowledge, the Besov regularity of $d$-dimensional L\'evy white noises has never been addressed in all generality.  Kusuoka \cite{Kusuoka1982support} estimated the weighted Sobolev regularity of the Gaussian white noise, while Veraar  \cite{Veraar2010regularity} obtained complete results on the local Besov regularity of the Gaussian white noise. However, these works are based on intrinsic Gaussian methods and are not easily extended  to the non-Gaussian case. 
In  \cite{Fageot2015besov}, we derived new results on the Besov regularity of symmetric-$\alpha$-stable (S$\alpha$S) L\'evy white noises on the $d$-dimensional torus. 
This paper is an extension of \cite{Fageot2015besov} in two ways: (1) we consider L\'evy white noises over $\R^d$ and deduce the local results as corollaries, and (2) we extend the results from S$\alpha$S white noises to general L\'evy white noises.

Other important works on the Besov regularity of $1$-dimensional L\'evy processes shall be mentioned. The pioneer works concern the Brownian motion \cite{Benyi2011modulation,Ciesielski1993quelques,Roynette1993}; see also \cite{Nualart2003besov} for extensions to more general Gaussian processes, including the fractional Brownian motion. 
Stable L\'evy processes are being studied in \cite{Ciesielski1993quelques} and \cite{Rosenbaum2009first}. Note that Rosenbaum   \cite{Rosenbaum2009first} is using wavelet techniques similar to ours. 
The case of general L\'evy processes was extensively studied by Schilling, both in the local \cite{Schilling1997Feller} and weighted cases \cite{Schilling1998growth,Schilling2000function}. Herren obtained similar local results in \cite{Herren1997levy}. In particular, Schilling and Herren rely  on the Blumenthal-Getoor indices introduced in \cite{Blumenthal1961sample} for their generalization to non-stable processes. In addition, they extend their results to more general classes of Markov processes. Those indices also play a crucial role in the present study. For a comprehensive survey on the Besov regularity of L\'evy processes, we refer the reader to \cite{Bottcher2013levy}. 

\paragraph{Contributions and outline}
The paper is organized as follows.
\begin{itemize}
	\item In Section \ref{sec:maths}, we define L\'evy white noises in the framework of generalized random processes. We also introduce weighted Besov spaces, based on the wavelet methods developed in \cite{Triebel2008function}.
	\item We derive estimates of the moments of L\'evy white noises in Section \ref{sec:estimates}. Moment estimates are standard technical results to study  the regularity of L\'evy processes \cite{Deng2015shift}. Our contribution is to adapt this question to L\'evy white noises by considering general test functions.
	\item In Section \ref{sec:measurable}, we show that weighted Besov spaces are measurable for the cylindrical $\sigma$-field on the space $\S'(\R^d)$ of tempered generalized functions. More concretely, it allows us to address the question of the Besov regularity of any generalized random process. 
	\item Sections \ref{sec:sobolev} and \ref{sec:weighted} are dedicated to the weighted Besov regularity of $d$-dimensional L\'evy white noises. Section \ref{sec:sobolev} is restricted to weighted $L_2$-Sobolev spaces, and is therefore an intermediate step  required to apply our wavelet techniques in Section \ref{sec:weighted}. The main result of this paper is Theorem \ref{theo:weighted}. There, we give sufficient conditions for a L\'evy white noise to be almost surely in a given weighted Besov space  in terms of the Blumenthal-Getoor indices of the noise.
	\item In Section \ref{sec:local}, we deduce the local Besov regularity of L\'evy white noises from the above.
	\item Finally, we discuss our results in Section \ref{sec:examples}  and apply them to some specific families of L\'evy white noises that are often encountered in practice.
\end{itemize}

\section{Preliminaries} \label{sec:maths}

	\subsection{Generalized Processes and L\'evy White Noises} \label{subsec:GRP}
		
The stochastic processes of this paper are defined in the framework of generalized random processes \cite[Chapter 3]{GelVil4}. It allows one in particular to consider L\'evy white noises as well-defined random processes, which is not possible in more traditional approaches since they do not admit a pointwise interpretation.

The Schwartz space of infinitely smooth functions and rapidly decaying functions on $\R^d$ is denoted by $\S(\R^d)$.  It is endowed with the topology associated with the  following notion of convergence: A sequence $(\varphi_n)$ of functions in $\S(\R^d)$ converges to $\varphi \in \S(\R^d)$ if, for every multiindex $\bm{\alpha} \in \N^d$ and every $\mu \geq 0$, the functions $\bm{x} \mapsto  \lvert \bm{x} \rvert^{\mu} \Der^{\bm{\alpha}} \{\varphi_n\}(\bm{x})$ converge to $\bm{x} \mapsto  \lvert \bm{x} \rvert^{\mu} \Der^{\bm{\alpha}} \{\varphi\}(\bm{x})$ in $L_2(\R^d)$,
where $\lvert \cdot \rvert$ is the Euclidian norm on $\R^d$.
The space $\S(\R^d)$ is a nuclear Fr\'echet space \cite[Section 51]{Treves1967}.
The topological dual of $\S(\R^d)$ is the  space $\S'(\R^d)$ of tempered generalized functions. A \emph{cylindrical set} of $\S'(\R^d)$ is a subset of the form
\begin{equation} \label{eq:cylinders}
	\Big\{ u \in \S'(\R^d), \quad \left( \langle u,\varphi_1 \rangle, \ldots, \langle u, \varphi_n\rangle \right) \in B \Big\},
\end{equation}
where $n\geq 1$, $\varphi_1, \ldots, \varphi_n \in \S(\R^d)$, and $B$ is a Borel subset of $\R^n$. We denote by $\mathcal{B}_c(\S'(\R^d))$ the \emph{cylindrical $\sigma$-field} of $\S'(\R^d)$, defined as the $\sigma$-field generated by the cylindrical sets.  Then, $(\S'(\R^d), \mathcal{B}_c(\S'(\R^d))$ is a measurable space. We fix the probability space $(\Omega,\mathcal{F},\mathscr{P})$.

\begin{definition}
	A \emph{generalized random process} is a measurable function 
	\begin{equation}
	s \ : \ (\Omega,\mathcal{F}) \rightarrow (\S'(\R^d),\mathcal{B}_c(\S'(\R^d)).
	\end{equation}
	Its \emph{probability law} is the measure on $\S'(\R^d)$, image of $\mathscr{P}$ by $s$. For  every $B\in \mathcal{B}_c(\S'(\R^d))$,
	\begin{equation}
		\mathscr{P}_s (B) = \mathscr{P} ( \{ \omega \in \Omega, s(\omega) \in B \}).
	\end{equation} 
	The \emph{characteristic functional} of $s$ is defined for every $\varphi \in \S(\R^d)$ by
	\begin{equation}
		\CF_s(\varphi) = \int_{\S'(\R^d)} \mathrm{e}^{\mathrm{i} \langle u,\varphi \rangle} \drm \mathscr{P}_s (u).
	\end{equation} 
\end{definition}

A generalized random process is a random element of the space of tempered generalized functions. The characteristic functional is the infinite-dimensional generalization of the characteristic function. It characterizes the law of $s$ in the sense that
\begin{equation}
	\mathscr{P}_{s_1} = \mathscr{P}_{s_2} \Leftrightarrow \CF_{s_1} =  \CF_{s_2},
\end{equation}
which  we denote by $s_1 \overset{(d)}{=} s_2$ (where $(d)$ stands for equality in distribution).
Since the space $\S(\R^d)$ is nuclear, the Minlos-Bochner theorem \cite{GelVil4,minlos1959generalized} gives a complete characterization of admissible characteristic functionals.

\begin{theorem}[Minlos-Bochner theorem]\label{theo:MB}
	A functional $\CF$ on $\S(\R^d)$ is the characteristic functional of a generalized random process $s$ if and only if it is continuous and positive-definite over $\S(\R^d)$ and satisfies $\CF(0) = 1$. 
\end{theorem}

L\'evy processes are random processes index by $\R$ with stationary and independent increments. They are deeply related to infinitely divisible random variables \cite{Sato1994levy}. For the same reasons, there is a one-to-one correspondence between infinitely divisible laws and L\'evy white noises. An infinitely divisible random variable $X$ can be decomposed as $X = X_1 + \cdots + X_n$ for every $n \geq 1$. We say that a function $f$ from $\R$ to $\C$ is a \emph{L\'evy exponent} if it is the continuous log-characteristic function of an infinitely divisible random variable. We say moreover that $f$ \emph{satisfies the Schwartz condition} if the moment $\mathbb{E}[|X|^\epsilon]$ of $X$ is finite for some $\epsilon >0$. 

Let $\bm{X} = (X_1, \ldots ,X_n)$ be i.i.d. infinitely divisible random variables with common L\'evy exponent $f$. By independence, the characteristic function of $\bm{X}$ is
\begin{equation} \label{eq:CFvector}
	\Phi_{\bm{X}} (\bm{\xi}) = \exp \left( \sum_{i=1}^n f(\xi_i) \right)
\end{equation}
for every $\bm{\xi} =(\xi_1,\ldots ,\xi_n ) \in \R^n$. The class of L\'evy white noise can be seen as the generalization of this principle in the continuous domain, up to the replacement of the sum in \eqref{eq:CFvector} by an integral.

\begin{definition}
	A \emph{L\'evy white noise} is a generalized random process $w$ with characteristic functional of the form
	\begin{equation} \label{eq:CFnoise}
		\CF_w(\varphi) =\exp\left(  \int_{\R^d} f(\varphi(\bm{x})) \drm \bm{x} \right)
	\end{equation}
	for every $\varphi \in \S(\R^d)$, where $f$ is a L\'evy exponent that satisfies the Schwartz condition.
\end{definition}

Gelfand and Vilenkin have proved that the functional \eqref{eq:CFnoise} is a valid characteristic functional on $\D(\R^d)$, the space of compactly supported and infinitely smooth functions, without the Schwartz condition on $f$ \cite{GelVil4}. The Schwartz condition is sufficient to extend this result to $\S(\R^d)$ \cite[Theorem 3]{Fageot2014}. Recently, Dalang and Humeau have shown that this condition is also necessary: A white noise with L\'evy exponent that does not satisfy the Schwartz condition is almost surely not in $\S'(\R^d)$ \cite[Theorem 3.13]{Dalang2015Levy}.

A L\'evy white noise is stationary, in the sense that $w \overset{(d)}{=} w(\cdot - \bm{x}_0)$ for every $\bm{x}_0 \in \R^d$. It is moreover independent at every point, meaning that $\langle w,\varphi \rangle$ and $\langle w,\psi \rangle$ are independent whenever $\varphi$ and $\psi \in \S(\R^d)$ have disjoint supports.
In $1$-D, we recover the usual notion of white noise, since $w$ is the derivative in the sense of generalized functions of the L\'evy process with the same L\'evy exponent. This principle can be extended in any dimension $d\geq 2$: The $d$-dimensional L\'evy white noise is the weak derivative $\Der_{x_1} \cdots \Der_{x_d} \{s\}$ of the $d$-dimensional L\'evy sheet $s$ \cite{Dalang2015Levy}. 
	
	\subsection{Weighted Sobolev and Besov Spaces} \label{subsec:besov}

Our goal is to characterize the smoothness of L\'evy white noises in terms of weighted Besov spaces. 
All our results related to Besov spaces would require the corresponding intermediate result for Sobolev spaces which we introduce in Section \ref{subsub:sobolev}.

		\subsubsection{Weighted Sobolev Spaces} \label{subsub:sobolev}

We set $\langle \bm{x} \rangle = \sqrt{1 + \lvert \bm{x} \rvert^2}$. The Fourier transform of $f \in \S'(\R^d)$ is denoted by $\widehat{f}$. For $\tau \in \R$, we define $\Lop_{\tau}$ (the Bessel operator of order $\tau$) as the pseudo-differential operator with Fourier multiplier $\langle \cdot \rangle^{\tau}$. In Fourier domain, we write
\begin{equation}
	\widehat{\Lop_\tau \{\varphi \}} ( \bm{\omega} ) := \langle \bm{\omega} \rangle^{\tau} \widehat{\varphi} (\bm{\omega})
\end{equation}
for every $\bm{\omega} \in \R^d$ and $\varphi \in \S(\R^d)$. When $\tau > 0$, the operator $\mathrm{I}_\tau = \Lop_{-\tau}$ is called a Bessel potential \cite{Grafakos2004classical}. The operator $\Lop_\tau$ is self-adjoint, linear, and continuous from $\S(\R^d)$ to $\S(\R^d)$,  since its Fourier multiplier is infinitely smooth and bounded by a polynomial function.
It can therefore be extended as a linear and continuous operator from $\S'(\R^d)$ to $\S'(\R^d)$. 

\begin{definition}	
	Let $\tau , \mu \in \R$.
	The \emph{Sobolev space of smoothness $\tau$} is defined by
	\begin{equation}
		W_2^{\tau}(\R^d) := \left\{  f \in \S'(\R^d), \quad  \Lop_\tau \{f\} \in L_2(\R^d)  \right\}
	\end{equation}
	and the \emph{Sobolev space of smoothness $\tau$ and weight $\mu$} is
	\begin{equation}
	W_2^{\tau}(\R^d; \mu) := \left\{  f \in \S'(\R^d),  \quad  \langle \cdot \rangle^{\mu} f  \in W_2^{\tau}(\R^d) \right\}.
	\end{equation}	
	We also set  $L_2(\R^d;\mu) := W_2^{0} (\R^d;\mu)$. 
\end{definition}

We summarize now the basic properties on weighted Sobolev spaces that are useful for our work.

\begin{proposition} \label{property_sobolev}
The following properties hold for weighted Sobolev spaces.
\begin{itemize}
	\item For $\mu , \tau \in \R$, $W_2^{\tau}(\R^d;\mu)$ is a Hilbert space for the scalar product
	\begin{equation}
		\langle f, g \rangle_{W_2^\tau(\R^d; \mu)} := \left\langle \Lop_\tau \{  \langle \cdot \rangle^{\mu}  f \} ,  \Lop_\tau \{  \langle \cdot \rangle^{\mu} g \} \right\rangle_{L_2(\R^d)}.
	\end{equation}
	We denote by $\lVert f \rVert_{W_2^\tau(\R^d;\mu)} = \langle f,f \rangle^{1/2}_{W_2^\tau(\R^d; \mu)}$ the corresponding norm.
	\item For $\mu\in \R$ fixed and for every $\tau_1 \leq \tau_2$, we have the continuous embedding
	\begin{equation} \label{eq:sobembtau} W_2^{\tau_2} (\R^d;\mu) \subseteq W_2^{\tau_1}(\R^d;\mu). \end{equation}
	\item For $\tau \in \R$ fixed and for every $\mu_1 \leq \mu_2$, we have the continuous embedding
	\begin{equation} \label{eq:sobembmu} W_2^{\tau} (\R^d;\mu_2) \subseteq W_2^{\tau}(\R^d;\mu_1). \end{equation}
	\item For $\mu , \tau \in \R$, the operator $\Lop_{\tau,\mu} : f  \mapsto 	\langle \cdot \rangle^{\mu} \Lop_\tau \{f\}$
	is an isometry from  $L_2(\R^d)$ to $W_2^{-\tau} (\R^d;-\mu) $.
	\item The dual space of $W_2^{\tau}(\R^d;\mu)$ is $W_2^{-\tau}(\R^d;-\mu)$ for every $\tau,\mu \in \R$. 
	\item We have the countable projective limit 
	\begin{equation} \label{eq:projective}
	\S(\R^d) = \bigcap_{\tau, \mu \in \R} W_2^{\tau}(\R^d;\mu) = \bigcap_{n \in \mathbb{N}} W_2^{n}(\R^d;n).
	\end{equation}
	\item We have the countable inductive limit 
	\begin{equation}    \label{eq:inductive}
	\S'(\R^d)  = \bigcup_{\tau, \mu \in \R} W_2^{\tau}(\R^d;\mu) = \bigcup_{n \in \mathbb{N}} W_2^{-n}(\R^d;-n).	
	\end{equation}
\end{itemize}
\end{proposition}

\begin{proof}
	The space $W_2^{\tau} (\R^d; \mu)$ inherits the Hilbertian structure of  $L_2(\R^d)$. For $\tau_1 \leq \tau_2$ and $\mu_1 \leq \mu_2$,  we have moreover the inequalities, 
	\begin{align}
		\lVert f \rVert_{W_2^{\tau_1}(\R^d; \mu)}   &\leq 	\lVert f \rVert_{W_2^{\tau_2}(\R^d; \mu)}, \\
		\lVert f \rVert_{W_2^\tau(\R^d; \mu_1)}     &\leq 	\lVert f \rVert_{W_2^\tau(\R^d; \mu_2)}, \label{eq:embeddingweights}
	\end{align}
from which we deduce \eqref{eq:sobembtau} and \eqref{eq:sobembmu}. The relation
\begin{equation}
	\lVert \Lop_{\tau,\mu} f \rVert_{W_2^{-\tau}(\R^d; - \mu)} = \lVert \Lop_{-\tau}  \{ \langle \cdot \rangle^{-\mu}  \Lop_{\tau,\mu} f \} \rVert_{L_2(\R^d)} =   \lVert f \rVert_{L_2(\R^d)}
	\end{equation}
	proves that $\Lop_{\tau,\mu}$ is an isometry. For every $f,g \in L_2(\R^d)$, we have that 
	\begin{equation} \label{eq:newscalarproduct}
	\langle \Lop_{\tau} \{\langle \cdot \rangle^{\mu} f\} ,  \Lop_{- \tau} \{\langle \cdot \rangle^{- \mu} g\} \rangle_{L_2(\R^d)} = \langle f,g \rangle_{L_2(\R^d)}.
	\end{equation} 
Since $W_2^{\tau}(\R^d;\mu) = \{ 	 \Lop_{\tau} \{\langle \cdot \rangle^{\mu} f\}, \ f\in L_2(\R^d)\}$, we easily deduce the dual of $W_2^{\tau}(\R^d;\mu)$ from \eqref{eq:newscalarproduct}. Finally, we can reformulate the topology on $\S(\R^d)$ as \eqref{eq:projective}. This implies directly \eqref{eq:inductive}. 
\end{proof}

		\subsubsection{Weigthed Besov Spaces}

Following Triebel \cite{Triebel2008function}, our definitions of weighted Besov spaces are based on wavelets. More traditionally, Besov spaces are introduced through the Fourier transform, see for instance \cite{Triebel1978interpolation}. The use of wavelets is equivalent and appears to be more convenient for our framework. 

Let us first introduce the wavelet bases that we should use.
We denote by $j\geq 0$ the scaling index and $\bm{m} \in \Z^d$ the shifting index.
Consider $\psi_{\mathrm{F}}$ and $\psi_{\mathrm{M}}$, which are the father and mother wavelet of a wavelet basis for $L_2(\R)$, respectively. We set $\mathrm{G}^0 = \{ M,F \}^d$ and $\mathrm{G}^j = \mathrm{G}^0 \backslash (F,\ldots, F)$ for $j \geq 1$. For a gender $G = (G_1,\ldots ,G_d) \in \mathrm{G}^0$ and for every $\bm{x} = (x_1,\ldots ,x_d) \inÊ\R^d$, we define
\begin{equation} \label{eq:separablewavelet}
	\psi_{G}(\bm{x}) = \prod_{i=1}^d \psi_{G_i}(x_i).
\end{equation}
 
\begin{proposition}[Section 1.2.1, \cite{Triebel2008function}] \label{prop:existencewavelet}
	For every integer $u \geq 0$, there exist compactly supported wavelets $\psi_{\mathrm{F}}$ and $\psi_{\mathrm{M}}$ such that 
\begin{equation}
	\big\{ \psi_{j,G,\bm{m}}, \quad  j \geq 0, G\in \mathrm{G}^j, \bm{m} \in \Z^d \big\}
\end{equation}
is an orthonormal basis of $L_2(\R^d)$, 
where 
\begin{equation}
\psi_{j, G,\bm{m}}  := 2^{jd/2} \psi_{G} (2^j \cdot - \bm{m})
\end{equation}
and $\psi_{G}$ is defined according to \eqref{eq:separablewavelet}. 
\end{proposition}

Concretely, \cite{Triebel2008function} considers separable Daubechies wavelets with the adequate regularity. 
For $\tau,\mu \in \R$ and $0< p,q \leq \infty$, the \emph{Besov sequence space $b_{p,q}^\tau(\mu)$} is the collection of sequences 
\begin{equation}
	\bm{\lambda} = \{ \lambda_{j,G,\bm{m}}, \quad j\geq 0, G \in \mathrm{G}^j, \bm{m} \in\Z^d \}
\end{equation}
such that
\begin{equation} \label{eq:besovsequencenorm}
	\lVert \bm{\lambda} \rVert_{b_{p,q}^\tau(\mu)} := \left( \sum_{j \geq 0} 2^{jq(\tau - d/p)} \sum_{G\in \mathrm{G}^j} \left(  \sum_{\bm{m}\in \Z^d}{ \langle 2^{-j} \bm{m} \rangle^{\mu p}} { \lvert \lambda_{j,G,\bm{m}}\rvert^p} \right)^{q/p}\right)^{1/q},
\end{equation}
with the usual modifications when $p$ and/or $q = \infty$.
 
\begin{definition} \label{def:besovclassical}
	Let $\tau, \mu \in \R$ and  $0<p,q \leq \infty$. 
	Fix
	\begin{equation} \label{eq:u}
	u > \max( \tau, d(1/p -1)_+ - \tau)
	\end{equation}
	and set $(\psi_{j,G,\bm{m}})$ a wavelet basis of $L_2(\R^d)$ with regularity $u$. 
	The \emph{weighted Besov space $B_{p,q}^{\tau}(\R^d;\mu)$} is the collection of generalized function $f \in \S'(\R^d)$ that can be written as
	\begin{equation} \label{eq:waveletdecomposition}
	 	f = \sum_{j,G,\bm{m}} 2^{- jd/2}\lambda_{j,G,\bm{m}} \psi_{j,G,\bm{m}}
	\end{equation}
	with $\bm{\lambda} = (\lambda_{j,G,\bm{m}}) \in b_{p,q}^\tau(\mu)$, where the convergence holds unconditionally in $\S'(\R^d)$. 
\end{definition} 

This definition is usually introduced as a characterization of Besov spaces. When \eqref{eq:waveletdecomposition} occurs, the representation is unique and  we have that  \cite[Theorem 1.26]{Triebel2008function}
\begin{equation}
\lambda_{j,G,\bm{m}} = 2^{jd/2} \langle f, \psi_{j,G,\bm{m}} \rangle.
\end{equation}  
To measure a given Besov regularity (fixed $p$, $q$, $\tau$, and $\mu$), we should select a wavelet with enough regularity   that the wavelet coefficients are well-defined for $f \in B_{p,q}^{\tau}(\R^d;\mu)$. This is the meaning of \eqref{eq:u}.  Under this condition, and for $f \in B_{p,q}^\tau(\R^d;\mu)$, the quantity
	\small
	\begin{equation} \label{eq:besovnorm}
	\lVert f \rVert_{B_{p,q}^{\tau}(\R^d;\mu)} := \left( \sum_{j\geq0} 2^{j(\tau - d/p + d / 2 )q} \sum_{G\in \mathrm{G}^j} \left( \sum_{\bm{m}\in\Z^d}  { \langle 2^{-j}\bm{m}\rangle^{\mu p}} { |\langle f, \psi_{j,G,\bm{m}} \rangle|^p}\right)^{q/p} \right)^{1/q} 
	\end{equation}
	\normalsize
	is finite, with the usual modifications when $p$ and/or $q = \infty$. The quantity \eqref{eq:besovnorm} is a norm for $p, q \geq 1$, and a quasi-norm otherwise. In any case, the Besov space is complete for its \mbox{(quasi-)norm}, and is therefore a \mbox{(quasi-)Banach} space.
	We have moreover the equivalence \cite[Theorem 4.2.2]{Edmunds2008function}
	\begin{equation}
		f \in 	B_{p,q}^\tau(\R^d;\mu) \Leftrightarrow  {\langle \cdot \rangle^{\mu}}f  \in B_{p,q}^\tau(\R^d) 
	\end{equation}
	with $B_{p,q}^\tau(\R^d) := B_{p,q}^\tau(\R^d;0)$ the classical (non-weighted) Besov space. The family of weighted Besov spaces includes the weighted Sobolev spaces due to the relation \cite[Section 2.2.2]{Edmunds2008function}
	\begin{equation}
		B_{2,2}^\tau(\R^d; \mu) = W_2^\tau(\R^d;\mu).
	\end{equation}
	
Weighted Besov spaces are embedded, as we show in Proposition \ref{prop:embeddingsbesov}.

\begin{proposition} \label{prop:embeddingsbesov}
We fix $\tau_0, \tau_1, \mu_0, \mu_1 \in \R$ and $0<p_0,q_0,p_1,q_1 \leq \infty$. We assume that 
\begin{equation}
\tau_0 > \tau_1 \text{ and } \mu_0 \geq \mu_1.
\end{equation}
If, moreover, we have that
\begin{equation} \label{eq:conditiontau}
p_0 \leq p_1 \text{ and }\tau_0 - \tau_1  \geq  d\left( \frac{1}{p_0} - \frac{1}{p_1} \right)
\end{equation}
or 
 \begin{equation} \label{eq:conditionmu}
 p_1 \leq p_0 \text{ and }\mu_0 - \mu_1  >  d\left( \frac{1}{p_1} - \frac{1}{p_0} \right),
 \end{equation}
then we have the continuous embedding
\begin{equation} \label{eq:besovembedding}
	B_{p_0,q_0}^{\tau_0} (\R^d;\mu_0) \subseteq B_{p_1,q_1}^{\tau_1} (\R^d;\mu_1).
\end{equation}
\end{proposition}

\begin{proof}
	Condition \eqref{eq:conditiontau} was proved to be sufficient in \cite[Section 4.2.3]{Edmunds2008function}. Note however that we can easily prove the embedding by using Besov sequence spaces, as we shall do   for the other condition. We could not find any precise statement of embeddings between Besov spaces for $p_1 \leq p_0$ in the literature, so we provide our own proof for the sufficiency of \eqref{eq:conditionmu}. 
	
	First of all, the parameter $q$ is dominated by parameters $\tau$ and $p$ in the sense that, for every $\tau \geq 0$, $\epsilon >0$, and  $0<p,q,r \leq \infty$, we have the embedding \cite[Proposition 2, Section 2.3.2]{Triebel2010theory}
	\begin{equation} \label{eq:qdominated}
		B_{p,q}^{\tau + \epsilon}(\R^d;\mu) \subseteq B_{p,r}^{\tau} (\R^d;\mu). 
	\end{equation}
	Note that Triebel considers unweighted spaces  in \cite{Triebel2010theory}, but the extension to the weighted case is obvious.
	This allows to restrict us to the case $q_0 = q_1 = q$. Fix $\bm{\lambda} = \{\lambda_{j,G,\bm{m}}, \ j\geq 0, G \in \mathrm{G}^j , \bm{m}\in \Z^d\}$. Due to the H\"older inequality, as soon as $1/a + 1/b = 1$, we have, for every $j \geq 0$ and $G \in \mathrm{G}^j$,  that
	\begin{equation} \label{eq:holderinterm}
		\sum_{\bm{m}\in\Z^d} \langle 2^{-j} \bm{m} \rangle^{\mu_1 p_1} \lvert \lambda_{j,G,\bm{m}} \rvert^{p_1} \leq \left( \sum_{\bm{m}\in \Z^d} \langle 2^{-j} \bm{m} \rangle^{(\mu_1-\mu_0) p_1 b} \right)^{1/b}  \left( 	\sum_{\bm{m}\in\Z^d}   \langle 2^{-j} \bm{m} \rangle^{\mu_0 p_1 a} \lvert \lambda_{j,G,\bm{m}} \rvert^{p_1 a }    \right)^{1/a}.
	\end{equation}
	We choose $a = p_0 / p_1 \geq 1$, thus $(\mu_1 - \mu_0) p_1 b = (\mu_1 - \mu_0) / (1/p_1 - 1/p_0) < -d$ by using \eqref{eq:conditionmu}, and $ \sum_{\bm{m}\in \Z^d} \langle 2^{-j} \bm{m} \rangle^{(\mu_1-\mu_0) p_1 b} < \infty$. Since $a p_1 = p_0$, we can  rewrite \eqref{eq:holderinterm} as
	\begin{equation}
		\left( \sum_{\bm{m}\in\Z^d} \langle 2^{-j} \bm{m} \rangle^{\mu_1 p_1} \lvert \lambda_{j,G,\bm{m}} \rvert^{p_1}\right)^{1/p_1} \leq C  \left( 	\sum_{\bm{m}\in\Z^d}   \langle 2^{-j} \bm{m} \rangle^{\mu_0 p_0} \lvert \lambda_{j,G,\bm{m}} \rvert^{p_0 }    \right)^{1/p_0}
	\end{equation}
	with $C>0$ a finite constant. Using \eqref{eq:besovsequencenorm}, this  implies   that $\lVert \bm{\lambda} \rVert_{b_{p_1,q}^{\tau_1}(\mu_1)} \leq C' \lVert \bm{\lambda} \rVert_{b_{p_0,q}^{\tau_0}(\mu_0)}$ and, then, the corresponding embedding between Besov sequence spaces. Finally, \eqref{eq:besovembedding} is a consequence of the isomorphism between Besov sequence spaces and Besov function spaces in Definition \ref{def:besovclassical} (see  \cite[Theorem 1.26]{Triebel2008function} for more details on the isomorphism). We let the reader adapt the proof when $p$ and/or $q$ are infinite.
\end{proof}

If the only knowledge provided to us is  that the  generalized function $f$ is in $\S'(\R^d)$, then this is not enough to of which regularity $u$ should be the wavelet to characterize the Besov smoothness of $f$. But if we have additional information on $f$, for instance its inclusion in a Sobolev space, then the situation is different. Proposition \label{prop:sobotobesov} gives a wavelet-domain criterion to determine if a generalized function $f$, known to be in $W_2^{\tau_0}(\R^d;\mu_0)$, is actually in a given Besov space $B_{p,q}^{\tau}(\R^d;\mu)$. 

\begin{proposition} \label{prop:sobotobesov}
	Let $\tau, \tau_0, \mu, \mu_0 \in \R$ and $0<p,q \leq \infty$. We set 
	\begin{equation} \label{eq:conditionu}
		u > \max ( \abs{\tau_0} , \abs{\tau - d(1/p - 1/2)_+}) .
	\end{equation}
	Then, the generalized function $f \in W_2^{\tau_0} (\R; \mu_0)$ is in $B_{p,q}^\tau(\R^d;\mu)$ if and only if
	\begin{equation} \label{eq:sobotobesov}
	\sum_{j\geq0} 2^{j(\tau - d/p + d / 2 )q} \sum_{G\in \mathrm{G}^j} \left( \sum_{\bm{m}\in\Z^d} { \langle 2^{-j}\bm{m}\rangle^{\mu p}} { |\langle f, \psi_{j,G,\bm{m}} \rangle|^p} \right)^{q/p} < \infty,
	\end{equation}
	with $(\psi_{j,G,\bm{m}})$ a wavelet basis of $L_2(\R^d)$ of regularity $u$, with the usual modifications when $p$ and/or $q = \infty$.
\end{proposition}

\begin{proof}
Let $\tau_1 < \min ( \tau_0, \tau - d(1 /p - 1/2)_+)$ and $\mu_1 \leq \min ( \mu_0 , \mu - d(1 /p - 1/2)_+)$. Then, according to Proposition \ref{prop:embeddingsbesov}, we have the embeddings
$$B_{p,q}^\tau (\R^d;\mu) \subseteq W_2^{\tau_1}(\R^d;\mu_1) \text{ and } W_2^{\tau_0}(\R^d;\mu_0)  \subseteq W_2^{\tau_1}(\R^d;\mu_1).$$
Condition \eqref{eq:conditionu} implies that we can apply Definition \ref{def:besovclassical} into the space $W_2^{\tau_1}(\R^d;\mu_1)$. In particular, if $(\psi_{j,G,\bm{m}})$ is the wavelet basis of Definition \ref{def:besovclassical} with regularity $u$, and for every function $f \in W_2^{\tau_1}(\R^d;\mu_1)$, then the wavelet coefficients $\langle f,\psi_{j,G,\bm{m}}\rangle$ are well-defined. Moreover, we have the characterization 
$$f \in B_{p,q}^{\tau} (\R^d;\mu) \Leftrightarrow \lVert f\rVert_{B_{p,q}^\tau(\R^d;\mu)}< \infty$$
for $f \in W_2^{\tau_1}(\R^d;\mu_1)$ and, therefore, for $f \in W_2^{\tau_1}(\R^d;\mu_0)$.
\end{proof}

\section{Moment Estimates for L\'evy White Noises} \label{sec:estimates}

The goal of this section is to obtain bounds for the $p$th moments of the random variable  $\langle w,\varphi\rangle$, where $w$ is a L\'evy white noise and $\varphi \in \S(\R^d)$.
The bounds are related to the moments of $\varphi$. For instance, for a symmetric $\alpha$-stable white noise $w_\alpha$,  it is known \cite[Lemma 2]{Fageot2015besov} that
\begin{equation} \label{eq:SalphaS_moments}
	\mathbb{E}\left[ \lvert \langle w_\alpha,\varphi \rangle\rvert^p \right] = C_{p,\alpha} \lVert \varphi \rVert_{\alpha}^p,
\end{equation}
where $C_{p,\alpha}$ is a constant, finite if and only if $\alpha = 2$ (Gaussian case), or $p < \alpha < 2$ (non-Gaussian case). 

	\subsection{Blumenthal-Getoor Indices}

In order to generalize \eqref{eq:SalphaS_moments} to non-stable white noises, we   consider the Blumenthal-Getoor indices introduced in \cite{Blumenthal1961sample}, which are   classical tools to estimate the moments of L\'evy processes \cite{Deng2015shift,Kuhn2015existence,Luschgy2008moment}. 
	
	\begin{definition} \label{def:BG}
		Let $f$ be a L\'evy exponent. 
		We consider the two intervals
		\begin{align}
			I_0 &= \left\{ p \in [0,2], \quad   \underset{|\xi| \rightarrow 0}{\lim \sup}  \frac{f(\xi)}{|\xi |^p} < \infty \right\},  \label{eq:I0} \\
			I_\infty &= \left\{ p \in [0,2], \quad  \underset{|\xi| \rightarrow \infty}{\lim \sup}  \frac{f(\xi)}{|\xi |^p} < \infty  \right\}. \label{eq:Iinfty}
		\end{align}
		The \emph{indices of Blumenthal-Getoor} are defined by
		\begin{equation}
			\beta_0  = \sup I_0, \quad  \beta_\infty  = \inf I_\infty.	
		\end{equation}
	\end{definition}
	
	
	\begin{proposition} \label{prop:boundexponent}
	Consider a L\'evy exponent $f$ with intervals $I_0$ and $I_\infty$   as in \eqref{eq:I0} and \eqref{eq:Iinfty}. Then,
	for $\tilde{\beta}_0 \in I_0$ and $\tilde{\beta}_\infty \in I_\infty$,  we have the inequality
	\begin{equation} \label{eq:boundfphi}
		\int_{\R^d} \lvert f(\varphi (\bm{x}) ) \rvert \drm \bm{x} \leq C \left( \lVert \varphi \rVert_{\tilde{\beta}_0}^{\tilde{\beta}_0} + \lVert \varphi \rVert_{\tilde{\beta}_\infty}^{\tilde{\beta}_\infty} \right)
	\end{equation}
	for all $\varphi \in L_{\tilde{\beta}_0}(\R^d) \bigcap  L_{\tilde{\beta}_\infty}(\R^d)$ and some constant $C>0$.
	\end{proposition}

	\begin{proof}
	The functions $\xi \mapsto \lvert f(\xi) \rvert   $ and $\xi \mapsto  \lvert \xi \rvert^{\tilde{\beta}_0}   +  \lvert \xi \rvert^{\tilde{\beta}_\infty}$ are two continuous functions, the second one dominating the first one at zero and at infinity. Therefore, there exists a constant $C>0$  that satisfies 
	\begin{equation*} \label{eq:boundexponent}
		\lvert f(\xi) \rvert \leq C \left(   \lvert \xi \rvert^{\tilde{\beta}_0}   +  \lvert \xi \rvert^{\tilde{\beta}_\infty} \right).
	\end{equation*} 
	Integrating the latter equation over $\bm{x} \in \R^d$ with $\xi = \varphi(\bm{x})$, we obtain \eqref{eq:boundfphi}.
	\end{proof}

	\subsection{Moment Estimates for $\langle w,\varphi\rangle$} \label{subsec:momentestimates}

We estimate the moments of a random variable by relaying the fractional moments to the characteristic function. Proposition \ref{prop:fractionalmoment} can be found for instance in \cite{Deng2015shift,Laue1980remarks,Matsui2013fractional} with some variations. For the sake of completeness, we recall the proof, similar to the one of \cite{Deng2015shift}. 

\begin{proposition} \label{prop:fractionalmoment}
	For a random variable $X$ with characteristic function $\Phi_X$ and $0<p<2$, we have the relation
	\begin{equation} \label{eq:pthmoments}
		\mathbb{E} [|X|^p] = c_p \int_{\R} \frac{1-\Re ( \Phi_X ) (\xi)}{|\xi|^{p+1}} \drm \xi \in [0,\infty],
	\end{equation}
	for some finite constant $c_p >0$, where $\Re ( z )$ denotes the real part of $z \in \C$. 
\end{proposition}

\begin{proof}
	For $p \in (0,2)$, we have, for every $x\in \R$,
	\begin{align}
		h(x) = \int_{\R} (1 - \cos (x\xi) ) \frac{\drm \xi}{|\xi|^{p +1}} = \left( \int_{\R} (1 - \cos (u) ) \frac{\drm u}{|u|^{p +1}} \right) |x|^p,
	\end{align}
	which is obtained by the change of variable $u = x \xi$. Applying this relation to $x=X$, we have, by Fubini's theorem and denoting $c_p = \left( \int_{\R} (1 - \cos (u) ) \frac{\drm u}{|u|^{p +1}} \right)^{-1}$, that
	\begin{align} 
		\mathbb{E} [|X|^p] & =c_p \mathbb{E} \left[ \int_{\R} (1 - \cos (\xi X) ) \frac{\drm \xi}{|\xi|^{p +1}} \right] \\
				& = c_p \int_{\R} ( 1 - \Re (\mathbb{E} [\mathrm{e}^{\mathrm{i} \xi X}] ) ) \frac{\drm \xi}{|\xi|^{p+1}} \\
				&= c_p \int_\R \frac{1-\Re ( \Phi_X ) (\xi)}{|\xi|^{p+1}} \drm \xi.
	\end{align}
\end{proof}

	\begin{theorem} \label{theo:moments}
	Consider a L\'evy white noise $w$ with  Blumenthal-Getoor indices $\beta_0$ and $\beta_\infty$.
	Then, for every $\varphi \in \S(\R^d)$ and $0<p<\beta_0$, we have the inequality
	\begin{equation} \label{eq:moments1}
		\mathbb{E} \left[  \lvert \langle w,\varphi \rangle\rvert^p \right]  \leq C \left(   \lVert \varphi \rVert_{\tilde{\beta}_0}^p+ \lVert \varphi \rVert_{\tilde{\beta}_\infty}^p   \right)
	\end{equation}	
	for some constant $C>0$, with $\tilde{\beta}_0 \in I_0$, $\tilde{\beta}_\infty \in I_\infty$, and $p< \tilde{\beta}_0, \tilde{\beta}_\infty$. 
Moreover, the result is still valid for $p=\tilde{\beta}_0 = \tilde{\beta}_\infty=2$
if $\beta_0=2 \in I_0$ (finite-variance case). 
	\end{theorem}

We   use  three simple lemmas in the proof of Theorem \ref{theo:moments}.

\begin{lemma} \label{lemma:1}
	There exists a constant $C>0$ such that, for every $z\in \C$ with $\Re ( z ) \leq 0$, we have that
	\begin{equation} \label{eq:boundexponential}
		|1 - \mathrm{e}^{z} | \leq C \left(1 - \mathrm{e}^{-\lvert z \rvert} \right). 
	\end{equation}
\end{lemma}

\begin{proof}
First, for any $z \in \C$, we have that  $(1 - \mathrm{e}^{-|z|} ) \geq 0$. Consider the function 
\begin{equation}
h(z) = \frac{\vert 1 - \mathrm{e}^{z} \vert}{1- \mathrm{e}^{-\vert z \vert}},
\end{equation}
which is well-defined and continuous over $\C\backslash\{0\}$. Since $h(z)  {\rightarrow} 1$ when ${z \rightarrow 0}$, we extend $h$ continuously to $\C$ with $h(0) = 1$. 
For $z$ such that $\Re(z) \leq 0$ and $|z| \geq \log 2$, we have $(1 - \mathrm{e}^{-|z|}) \geq 1/2$ and $|1-\mathrm{e}^{z}| \leq 2$, hence
$0 \leq h(z) \leq 4$. 
Moreover, $$K = \{ z\in \C, \quad  |z| \leq \log 2 , \  \Re(z) \leq 0\}$$ is compact. Therefore, the continuous function $h$ is bounded on $K$. Finally, we have shown that $h$ is bounded on $\{ z \in \C, \Re(z) \leq 0\}$, which gives \eqref{eq:boundexponential}. 
\end{proof}

\begin{lemma} \label{lemma:2}
	Let $x,y \leq 1$. Then, $(1 - xy ) \leq (1-x) + (1-y)$.
\end{lemma}

\begin{proof}
	From the assumption, we know that $0 \leq (1-x) (1-y) = (1-x) + (1-y) - (1-xy)$ and the result follows.
\end{proof}

\begin{lemma} \label{lemma:3}
 For every $\alpha \in (0,2)$ and $p < \alpha$,  there exists a constant $c_{p,\alpha}$ such that
 \begin{equation}
	\int_{\R} \frac{1 - \mathrm{e}^{- | x \xi |^\alpha}} {|\xi|^{p+1}} \drm \xi = c_{p,\alpha} |x|^p
\end{equation}
for every $x \in \R$.
\end{lemma}

\begin{proof}
The result is deduced by the change of variable $u = x \xi$ in the integral.
\end{proof}

\begin{proof}[Proof of Theorem \ref{theo:moments}]
Defining $X = \langle w,\varphi \rangle$, the characteristic function of $X$ is
\begin{equation}
	\Phi_X(\xi) = \exp \left( \int_{\R^d} f(\xi \varphi (\bm{x}) ) \drm \bm{x}\right).
\end{equation}
Moreover, from Proposition \ref{prop:boundexponent}, we have that
\begin{equation} \label{eq:boundgeneralizedexponent}
\int_{\R^d} |f(\xi \varphi)| \leq C   \left( \lVert \varphi \rVert_{\tilde{\beta}_0}^{\tilde{\beta}_0} \lvert \xi \rvert^{\tilde{\beta}_0}   + \lVert \varphi \rVert_{\tilde{\beta}_\infty}^{\tilde{\beta}_\infty} \lvert \xi \rvert^{\tilde{\beta}_\infty}  \right).
\end{equation}
We have, therefore, that
\begin{align}
	1 - \Re (\Phi_X )(\xi) &\leq | 1 - \Phi_X(\xi) | \nonumber \\
		& \overset{(i)}{\leq} C \left( 1 - \exp \left( - \left\lvert \int f(\xi \varphi) \right\lvert \right) \right)  \nonumber \\
		& \overset{(ii)}{\leq} C \left( 1 - \exp \left( -  \int | f(\xi \varphi) | \right) \right) \nonumber  \\
		& \overset{(iii)}{\leq} C' \left( 1 - \exp(  - \lVert \varphi \rVert_{\tilde{\beta}_0}^{\tilde{\beta}_0} \lvert \xi \rvert^{\tilde{\beta}_0}  ) \exp(-  \lVert \varphi \rVert_{\tilde{\beta}_\infty}^{\tilde{\beta}_\infty} \lvert \xi \rvert^{\tilde{\beta}_\infty} )  \right) \nonumber \\
		&  \overset{(iv)}{\leq} C' \left(  \left( 1 - \exp(  - \lVert \varphi \rVert_{\tilde{\beta}_0}^{\tilde{\beta}_0} \lvert \xi \rvert^{\tilde{\beta}_0}  )  \right) +  \left( 1 - \exp(-  \lVert \varphi \rVert_{\tilde{\beta}_\infty}^{\tilde{\beta}_\infty} \lvert \xi \rvert^{\tilde{\beta}_\infty} )  \right) \right),  \label{eq:interm}  
\end{align}
where (i) comes from Lemma \ref{lemma:1}, (ii) and (iii) from the fact that $x \mapsto 1 - \mathrm{e}^{-x}$ is increasing, (iii) from \eqref{eq:boundgeneralizedexponent}, and (iv) from Lemma \ref{lemma:2}. 
Finally, by applying Lemma \ref{lemma:3} with $x = \lVert \varphi \rVert_{\tilde{\beta}_0},\alpha = \tilde{\beta}_0$  and $x = \lVert \varphi \rVert_{\tilde{\beta}_\infty}, \alpha = \tilde{\beta}_\infty$, respectively, we deduce using \eqref{eq:pthmoments}  that

	\begin{equation}
		\mathbb{E} [|X|^p] = c_p \int_{\R} \frac{1-\Re( \Phi_X ) (\xi)}{|\xi|^{p+1}} \drm \xi \leq  C''\left(  \lVert \varphi \rVert_{\tilde{\beta}_0}^p+ \lVert \varphi \rVert_{\tilde{\beta}_\infty}^p \right),
	\end{equation}
	ending the proof. 
	
	The finite-variance case, when $\beta_0 = 2 \in I_0$, cannot be deduced with the same arguments, since \eqref{eq:pthmoments} is not valid any more. However, we know in this case that 
	\begin{equation}
	\mathbb{E}[\langle w, \varphi \rangle^2] = \sigma^2 \lVert \varphi \rVert_2^2 + \mu^2 \left( \int_{\R^d} \varphi \right)^2 \leq  \sigma^2 \lVert \varphi \rVert_2^2 + \mu^2 \rVert \varphi \lVert_1^2,
	\end{equation}
	where $\sigma^2$ and $\mu$ are   the variance and the mean of the infinitely divisible random variable with the same L\'evy exponent as $w$ \cite[Proposition 4.15]{Unser2014sparse}, respectively. Hence, the result is still valid.	
\end{proof}

We  take advantage of Theorem \ref{theo:moments} in a slightly less general form and apply it to wavelets, which are functions that are rescaled versions of an initial function.Corollary \ref{coro:momentbound} is a useful consequence of Theorem \ref{theo:moments}. For $\varphi \in \S(\R^d)$, $j \geq 0$, and $\bm{m} \in \Z^d$, we set $\varphi_{j,k} = 2^{jd/2} \varphi( 2^j \cdot - \bm{m})$. 

\begin{corollary} \label{coro:momentbound}
	Let $w$ be a L\'evy white noise with indices $\beta_0$ and $\beta_\infty$. 
	We assume either that $\beta_\infty < \beta_0$, or that $\beta_\infty = \beta_0 \in I_\infty \bigcap I_0$. 
	We fix $p < \beta \in I_0 \bigcap I_\infty$. 
	Then, there exists a constant $C$ such that, for every $\varphi \in \S(\R^d)$, $j \geq 0$, and $\bm{m}\in \Z^d$, 
	\begin{equation} \label{eq:boundwavelets}
		\mathbb{E} \left[  \lvert \langle w, \varphi_{j,\bm{m}}  \rangle\rvert^p \right]  \leq C 2^{jdp( 1/2 - 1/\beta)} \lVert \varphi \rVert_{\beta}^{p}.
	\end{equation}
	Moreover, this result is still valid if $p = \beta = 2 \in I_0$.
\end{corollary}

\begin{proof}
	Remark first that the assumptions on $\beta_0$ and $\beta_\infty$ imply that $I_0 \bigcap I_\infty \neq \emptyset$.
	We apply Theorem \ref{theo:moments} with $\tilde{\beta}_\infty = \tilde{\beta}_0 = \beta$. In particular, we have that $\mathbb{E} \left[  \lvert \langle w,\varphi_{j,\bm{m}} \rangle\rvert^p \right] \leq C   \lVert \varphi_{j,\bm{m}}  \rVert_{\beta}^{p}$. The result follows from the relation
	\begin{equation}
		  \lVert \varphi_{j,\bm{m}}  \rVert_{\beta}^{p} = 2^{jdp/2} \left( \int_{\R^d} \lvert \varphi ( 2^j \bm{x} -\bm{m} )  \rvert^\beta \drm \bm{x} \right)^{p/\beta} =  2^{jdp( 1/2 - 1/\beta)} \lVert \varphi \rVert_{\beta}^{p},
	\end{equation}
the last equality being obtained by the change of variable $\bm{y} = 2^{j} \bm{x} -\bm{m}$. The result is still valid for $p = \beta = 2$ for which we can still apply Theorem \ref{theo:moments}.
\end{proof}

	\subsection{Application of Moment Estimates to the Extension of $\langle w, \varphi\rangle$ for Non-Smooth Functions}

A generalized random process $s$ is a random variable from $\Omega$ to $\S'(\R^d)$. Alternatively, it can be seen as a linear and continuous map from $\S(\R^d)$ to the space $L_0(\Omega)$  of real random variables, that associates to $\varphi \in \S(\R^d)$ the random variable $\langle s,\varphi \rangle$\footnote{This is not as obvious as it might seem in infinite-dimension, and is again due to the nuclear structure of $\S(\R^d)$. For the links between $E'$-valued random variables and linear functionals from $E$ to $L_0(\Omega)$ (with $E'$ the dual of $E$), see \cite[Section 2.3]{Ito1984foundations}, in particular Theorems 2.3.1 and 2.3.2.}. The space $L_0(\Omega)$ is a Fr\'echet space associated with the following notion of convergence. A sequence $(X_n)$ converges to $X$ in $L_0(\Omega)$ if and only if  $\lVert X_n - X \rVert_{L_0(\Omega)} = \mathbb{E} [\min (\lvert X_n - X \rvert , 1)] \underset{n\rightarrow \infty}{\longrightarrow} 0$. The topology on $L_0(\Omega)$ corresponds to the convergence in probability. We also define the spaces $L_p(\Omega)$ for $0<p<\infty$ associated for $p\geq 1$ ($p<1$, respectively) with the  norm (the quasi-norm, respectively) $\lVert X\rVert_{L_p(\Omega)} = \mathbb{E}[\lvert X\rvert^p]$. See \cite[Section 2.2]{Ito1984foundations} for more details.

To measure the Besov regularity of L\'evy white noises, we shall consider random variables $\langle w,\varphi \rangle$ for test functions $\varphi$ not in $\S(\R^d)$. We handle this by extending  the domain of test functions through which one can observe a generalized random process.

\begin{lemma} \label{lemma:extendsvarphi}
	Let $0<p,\beta <\infty$. Consider a generalized random process $s$. We assume that, for all $\varphi \in \S(\R^d)$,
	\begin{equation} \label{eq:boundmoment}
		\mathbb{E} [\lvert \langle  s , \varphi \rangle \rvert^p] \leq C \lVert \varphi \rVert_{\beta}^p
	\end{equation}
	for some constant $C>0$. Then, we can extend $s$ as a linear and continuous map from $L_{\beta}(\R^d)$ to $L_0(\Omega)$. Moreover,   \eqref{eq:boundmoment} remains valid for $\varphi \in L_{\beta}(\R^d)$.
\end{lemma}

Before proving this result, we remark that it immediately  implies Corollary \ref{coro:extendedcoro1}.

\begin{corollary} \label{coro:extendedcoro1}
Under the conditions of Corollary \ref{coro:momentbound}, we can extend $\langle w,\varphi \rangle$ for $\varphi \in L_{\beta}(\R^d)$. Moreover, \eqref{eq:boundwavelets}  remains valid for any $\varphi \in L_{\beta}(\R^d)$. 
\end{corollary}

\begin{proof}[Proof of Lemma \ref{lemma:extendsvarphi}] 
Fix $\varphi \in L_\beta(\R^d)$. 
Let $(\varphi_n)$ be a sequence of elements in $\S(\R^d)$ such that 
$\varphi_n \underset{n \rightarrow \infty}{\longrightarrow} \varphi$ in $L_\beta(\R^d)$.
Such a sequence exists by density of $\S(\R^d)$ into $L_{\beta}(\R^d)$, well-known for $\beta \geq 1$, and easily extended for $0<\beta<1$. 
Then, for every $n,m \in \N$, it holds that
\begin{align*}
	\mathbb{E} [ \lvert \langle s,\varphi_n\rangle -  \langle s,\varphi_m\rangle  \lvert^p ] &= \mathbb{E} [ \lvert \langle s,\varphi_n - \varphi_m\rangle  \lvert^p ] \\
					&\leq C \lVert \varphi_n - \varphi_m \rVert^p_{\beta},
\end{align*}
so that $(\langle s, \varphi_n\rangle )$ is a Cauchy sequence in the complete space $L_p(\Omega)$. It therefore converges to a random variable $X$ in $L_p(\Omega)$.
Moreover, this limit does not depend on the sequence $(\varphi_n)$ chosen to approximate $\varphi$. Indeed, if $(\tilde{\varphi}_n)$ is another sequence converging to $\varphi$ in $L_{\beta}(\R^d)$, and defining $\tilde{X}$ as  the limit of $\langle s,\tilde{\varphi}_n\rangle$ in $L_p(\Omega)$, we have that
\begin{equation} \label{eq:splitXXtilde}
	\mathbb{E}[ \lvert X - \tilde{X} \rvert^p ]  = \mathbb{E} [ \lvert X - \langle s, \varphi_n \rangle +  \langle s, \varphi_n - \tilde{\varphi}_n \rangle +  \langle s, \tilde{\varphi}_n \rangle - \tilde{X} \rvert^p ] 
\end{equation} 
We have moreover, for $x,y,z \in \R$, that
\begin{align*}
	|x + y + z|^p \leq \begin{cases}
					3^{p-1} \left( |x|^p + |y|^p + |z|^p  \right) & \text{if} \  p \geq 1, \\
					 |x|^p + |y|^p + |z|^p  & \text{otherwise.}
				\end{cases}					 
\end{align*}
Hence, if we set  $C_p = 3^{p-1}$ if $p\geq 1$ and $C_p = 1$ if $p<1$, we deduce from \eqref{eq:splitXXtilde} that
\begin{equation} \label{eq:splitXXtilde2}
	\mathbb{E}[ \lvert X - \tilde{X} \rvert^p ]  \leq C_p \left(  \mathbb{E} [ \lvert X - \langle s, \varphi_n \rangle \rvert^p ]  + \mathbb{E} [  \langle s, \varphi_n- \tilde{\varphi}_n \rangle   \rvert^p ] + \mathbb{E} [   \langle s, \tilde{\varphi}_n \rangle -  \tilde{X} \rvert^p ] \right).
\end{equation} 
 The three terms in \eqref{eq:splitXXtilde2} vanish when $n\rightarrow \infty$. For the second one, we   use that $ \mathbb{E} [  \langle s, \varphi_n- \tilde{\varphi}_n \rangle   \rvert^p ] \leq C \lVert \varphi_n - \tilde{\varphi}_n \rVert_{\beta}^p$. This proves that $X = \tilde{X}$ almost surely.  We therefore define $X = \langle s,\varphi\rangle$.

Finally, \eqref{eq:boundmoment} is still valid for $\varphi \in L_{\beta}(\R^d)$ by continuity of $s$ from $L_\beta(\R^d)$ to $L_p(\Omega)$.
\end{proof}

\section{Measurability of Weighted Besov Spaces} \label{sec:measurable}

A generalized random process is a measurable function from $\Omega$ to $\S'(\R^d)$, endowed with the cylindrical $\sigma$-field $ \mathcal{B}_c(\S'(\R^d))$.
In the next sections, we  investigate  in which Besov space (local or weighted) is a given L\'evy white noise. 
We show here that this question is meaningful in the   sense that any Besov space $B_{p,q}^{\tau} (\R^d;\mu)$ is measurable in $\S'(\R^d)$. 

\begin{proposition} \label{prop:weighted_measurability}
For every $0<p,q\leq \infty$ and $\tau, \mu \in \R$, we have that
\begin{equation} \label{eq:BpqinBc}
	B_{p,q}^{\tau}(\R^d;\mu) \in \mathcal{B}_c(\S'(\R^d)).
\end{equation}
\end{proposition}

The proof of this result is very similar to the one of \cite[Theorem 4]{Fageot2015besov}, except we work now over $\R^d$ and deal with weights. In particular, we will rely on \cite[Lemma 1]{Fageot2015besov}. 

\begin{proof}
We obtain the desired result in three steps. 
We treat the case $p,q < \infty$ and let the reader adapt the proof for $p$ and/or $q=\infty$.
\begin{itemize}

\item First, we show that $W_2^{\tau}(\R^d;\mu) \in \mathcal{B}_c(\S'(\R^d))$ for every $\tau, \mu \in \R$. 
Let $(h_n)_{n\in \N}$ be an orthonormal basis of $L_2(\R^d)$, with $h_n \in \S(\R^d)$ for all $n\geq 0$. (We can for instance consider the Hermite functions, based on Hermite polynomials, see \cite[Section 2]{Simon2003distributions} or \cite[Section 1.3]{Ito1984foundations} for the definitions.) The interest of having basis functions in $\S(\R^d)$ is that we have the   characterization
\begin{equation}
	L_2(\R^d) = \left\{ f \in \S'(\R^d), \  \sum_{n \in \N} \lvert \langle f , h_n \rangle \rvert^2 < \infty \right\}.
\end{equation}
More generally, with the notations of Section \ref{subsub:sobolev},  $f \in W_2^{\tau}(\R^d;\mu)$ if and only if $\Lop_{\tau} \{ \langle \cdot \rangle^\mu f \} \in L_2(\R^d)$, from which we deduce that
\begin{equation}
	W_2^{\tau}(\R^d;\mu) = \left\{ f \in \S'(\R^d), \  \sum_{n \in \N} \lvert \langle f , \langle \cdot \rangle^{\mu} \Lop_{\tau} \{ h_n \} \rangle \rvert^2 < \infty \right\}.
\end{equation}
We can therefore apply \cite[Lemma 1]{Fageot2015besov} with $\alpha = 2$, $S =\N$, and $\varphi_n = \langle \cdot \rangle^{\mu} \Lop_{\tau} \{ h_n \}$, to deduce that $W_2^{\tau}(\R^d;\mu) \in \mathcal{B}_c(\S'(\R^d))$.

\item For any $\tau, \mu \in \R$,  the cylindrical $\sigma$-field of $W_2^{\tau}(\R^d;\mu)$ is the $\sigma$-field $\mathcal{B}_c(W_2^{\tau}(\R^d;\mu))$generated by the sets
\begin{equation}
	\Big\{ u \in W_2^{\tau}(\R^d;\mu), \ \left( \langle u,\varphi_1 \rangle, \cdots, \langle u, \varphi_n\rangle \right) \in B \Big\},
\end{equation}
where $N \geq 1$, $\varphi_1, \ldots, \varphi_N \in W_2^{-\tau}(\R^d;-\mu)$, and $B$ is a Borelian subset of $\R^N$. Then,   $W_2^{\tau}(\R^d;-\mu) \in \mathcal{B}_c(\S'(\R^d))$ implies that 
\begin{equation} \label{eq:BcinBc}
\mathcal{B}_c(W_2^{\tau}(\R^d;\mu)) \subset \mathcal{B}_c(\S'(\R^d)).
\end{equation}

\item Finally, we show that $B_{p,q}^\tau(\R^d;\mu) \in \mathcal{B}_c(W_2^{\tau_1}(\R^d;\mu_1))$ for some $\tau_1, \mu_1 \in \R$. Coupled with \eqref{eq:BcinBc}, we deduce \eqref{eq:BpqinBc}.

Fix $\tau_1 \leq \tau + d\left( 1/2 - 1/p\right)$ and $\mu_1 < \mu + d\left( 1/p - 1/2\right)$. According to Proposition \ref{prop:embeddingsbesov}, we have the embedding $B_{p,q}^\tau(\R^d;\mu) \subseteq W_2^{\tau_1} (\R^d;\mu_1)$. Now, we can rewrite Proposition \ref{prop:sobotobesov} (with $\tau_0 = \tau_1$ and $\mu_0 = \mu_1$) as 
\begin{equation}
	B_{p,q}^{\tau}(\R^d;\mu) = 
	\left\{ f \in W_2^{\tau_1}(\R^d;\mu_1), \ 
	\sum_{j,G} \left( \sum_{\bm{m}} \rvert \langle f , 2^{j (\tau - d/p + d/2)} \langle 2^{-j} \bm{m} \rangle^\mu \psi_{j,G,\bm{m}} \rangle \rvert^p \right)^{q/p} < \infty 
	\right\}. 
\end{equation}
Again, we   apply \cite[Lemma 1]{Fageot2015besov} with $S = \{  (j,G), j \in \Z, G \in \mathrm{G}_j \}$, $n= ( j,G)$, $T_n = T_{(j,G)} = \Z^d$, $\varphi_{n,m} = \varphi_{j,G,\bm{m}} = 2^{j (\tau - d/p + d/2)} \langle 2^{-j} \bm{m} \rangle^\mu \psi_{j,G,\bm{m}} \rangle$, $\alpha = p$, and $\beta = q/p$, to deduce that $B_{p,q}^\tau(\R^d;\mu) \in \mathcal{B}_c(W_2^{\tau_1}(\R^d;\mu_1))$. Remark that, strictly speaking, Lemma 1 in \cite{Fageot2015besov} is stated for $T_n$ finite, but the proof is easily adapted to $T_n$ countable.
\end{itemize}
\end{proof}

Proposition  \ref{prop:weighted_measurability}  suggests  that the framework of generalized random processes is particularly well-suited to addressing regularity issues. 
By comparison, we recall that the space $\mathcal{C}(\R^d)$ of continuous functions is not measurable with respect to the topological $\sigma$-field on the space of all functions, while we have that
\begin{equation}
	\mathcal{C}(\R^d) \in \mathcal{B}_c(\D'(\R^d)),
\end{equation}
which is the cylindrical $\sigma$-field of the space of generalized functions (not necessarily tempered) \cite[Proposition III.3.3]{Fernique1967processus}.
See \cite{Cartier1963processus} for a discussion on the measurability of function spaces and the advantages of generalized random processes.

\section{L\'evy White Noises on Weighted Sobolev Spaces} \label{sec:sobolev}

In order to characterize the Besov smoothness of L\'evy white noises, we first obtain information on their Sobolev smoothness. 

\begin{proposition} \label{prop:sobolevregularity}
A L\'evy white noise $w$ with indices $\beta_0>0$ and $\beta_\infty$ is in the weighted Sobolev space $W_{2}^{- \tau}(\R^d;- \mu)$ if
\begin{equation} \label{eq:weightedsobolev}
 \mu > \frac{d}{\beta_0} \text{ and } \tau > \frac{d}{2}.
\end{equation}
\end{proposition}

\begin{proof}
As we have seen in Proposition \ref{property_sobolev},  we have the countable projective limit 
 $$\S(\R^d) = \bigcap_{\tau, \mu \in \R} W_2^{\tau}(\R^d;\mu) = \bigcap_{n \in \N} W_2^{n}(\R^d;n).$$ We are therefore in the context of \cite[Theorem A.2]{Hida2004}. 
It implies in particular that, if
\begin{itemize}
	\item the characteristic functional $\CF_w$ of $w$ is continuous over $L_2(\R^d;\mu_0)$, and
	\item the identity operator $\mathrm{I}$ is Hilbert-Schmidt from $W_2^\tau(\R^d;\mu)$ to $L_2(\R^d;\mu_0)$,
\end{itemize}
then $w \in W_2^{-\tau}(\R^d,-\mu) = \left( W_2^\tau(\R^d;\mu) \right)'$ almost surely.

\paragraph{Continuity of $\CF_w$}
Let $\mu_0 > d\left( \frac{1}{\beta_0} - \frac{1}{2} \right) \geq 0$. Fix $\epsilon>0$ small enough such that 
\begin{equation} \label{eq:boundepsilon}
	\beta_0 - \epsilon > 0 \text{ and } \mu_0 > d\left( \frac{1}{{\beta_0} - \epsilon} - \frac{1}{2} \right).
\end{equation}
Applying Proposition \ref{prop:boundexponent} with $\tilde{\beta}_\infty = 2$ and $\tilde{\beta}_0 = \beta_0 -\epsilon$, we deduce that
\begin{equation}  \label{eq:interm1sobolev}
	\int_{\R^d} \lvert f(\varphi) \rvert \leq C \left( \lVert \varphi \rVert_{\tilde{\beta}_0}^{\tilde{\beta}_0} +  \lVert \varphi \rVert_2^2 \right).
\end{equation}
Since $\mu_0 >0$, we have that $ \lVert \varphi \rVert_{L_2(\R^d)}^2 \leq  \lVert \varphi \rVert_{L_2(\R^d;\mu_0)}^2$. Moreover, using H\"older inequality, we have
 \begin{equation}
\lVert \varphi \rVert_{\tilde{\beta}_0}^{\tilde{\beta}_0} = \int \lvert \varphi \rvert^{\tilde{\beta}_0} \leq \int ( \lvert \varphi  \rvert^{\tilde{\beta}_0}\langle \cdot \rangle^{\mu_0 \tilde{\beta}_0} )^p \int \langle \cdot \rangle^{- \mu_0 \tilde{\beta}_0 q} 
\end{equation}
for any $1 \leq p,q \leq \infty$ such that $1/p + 1/q = 1$. Setting $p= 2 / \tilde{\beta}_0 \geq 1$, we have $q = \frac{2}{2 - \tilde{\beta}_0}$. Therefore,
\begin{equation} \label{eq:interm2sobolev}
\lVert \varphi \rVert_{\tilde{\beta}_0}^{\tilde{\beta}_0} \leq \int \left( \lvert \varphi \rvert \langle \cdot \rangle^{\mu_0}\right)^2 \int {\langle \cdot \rangle^{-\frac{2\tilde{\beta}_0 \mu_0}{2 -\tilde{\beta}_0}} },
\end{equation} 
the last integral being finite due to \eqref{eq:boundepsilon}, which implies that $\frac{2\tilde{\beta}_0 \mu_0}{2 -\tilde{\beta}_0} > d$. Finally, injecting \eqref{eq:interm2sobolev} into  \eqref{eq:interm1sobolev}, we obtain the inequalities
\begin{equation}  \label{eq:interm3sobolev}
	\lvert \log \CF_w(\varphi) \rvert \leq \int_{\R^d} \lvert f(\varphi) \rvert \leq C'  \lVert \varphi \rVert_{L_2(\R^d;-\mu_0)}^2.
\end{equation}
This implies that $\CF_w$ is well-defined over $L_2(\R^d;-\mu_0)$ and continuous at $\varphi = 0$. Since $\CF_w$ is positive-definite, it is therefore continuous over $L_2(\R^d;-\mu_0)$ \cite{Horn1975quadratic}.

\paragraph{Hilbert-Schmidt condition}
The operator $\Lop_{-\tau, -\mu}$ is defined in Proposition \ref{property_sobolev}. Its kernel $k_{-\tau,-\mu}$, such that 
\begin{equation}
\Lop_{-\tau, -\mu} \{ \varphi \} (\bm{x}) = \int_{\R^d} k_{-\tau,-\mu} (\bm{x}, \bm{y}) \varphi (\bm{y}) \drm \bm{y},
\end{equation}
is given by
\begin{equation}
k_{-\tau,-\mu} (\bm{x}, \bm{y}) = \langle \bm{x} \rangle ^{- \mu}  k_{-\tau} (\bm{y} - \bm{x})
\end{equation}
with $ k_{-\tau}$ being the Fourier multiplier of $\Lop_{-\tau}$,  which satisfies $\widehat{k}_{-\tau} (\bm{\omega}) = \langle \bm{\omega} \rangle^{-\tau}$. 

According to Proposition \ref{property_sobolev}, the operator $\Lop_{-\tau,-\mu+ \mu_0}$ is an isometry from $L_2(\R^d)$ to $W_2^{\tau}(\R^d;\mu- \mu_0)$. Therefore, we have the equivalences
\begin{align} \label{eq:equivalenceHS}
	\mathrm{I} : W_2^{\tau}(\R^d;\mu) \rightarrow L_2(\R^d;\mu_0) \text{ is HS} & \Leftrightarrow \mathrm{I} : W_2^{\tau}(\R^d;\mu-\mu_0) \rightarrow L_2(\R^d;\mu_0)  \text{ is HS} \nonumber \\
	& \Leftrightarrow \Lop_{-\tau,-\mu+\mu_0} : L_2 (\R^d) \rightarrow L_2(\R^d) \text{ is HS}. 
\end{align}
There HS means Hilbert-Schmidt.
The last condition in \eqref{eq:equivalenceHS} is equivalent with the condition \cite[Theorem VI.23]{Reed1980methods}  
\begin{equation} 
	k_{- \tau, - \mu + \mu_0} (\cdot ,\cdot ) \in L_2(\R^d\times\R^d).
\end{equation}
Moreover,   by the change of variable $\bm{z} = (\bm{y}- \bm{x})$ and using Parseval relation, 
\begin{align}
	\int_{\R^d\times\R^d} k_{- \tau, - \mu+ \mu_0} (\bm{x},\bm{y})^2 \drm \bm{x} \drm \bm{y} &= \int_{\R^d} \langle \bm{x} \rangle^{-2(\mu-\mu_0)} \drm \bm{x} \int_{\R^d} k_{-\tau} (\bm{z}) \drm \bm{z} \\
				&= \lVert \langle \cdot \rangle^{-(\mu-\mu_0)} \rVert_{L_2(\R^d)}^2  \lVert \langle \cdot \rangle^{-\tau} \rVert_{L_2(\R^d)}^2.
\end{align}
Therefore, $k_{\tau, \mu-\mu_0}$ is square-integrable if and only if $(\mu - \mu_0) > d/2$ and $\tau > d/2$.
By choosing $\mu_0$ close to $d(1/\beta_0 - 1/2)$, we finally obtain that $w \in W_2^{-\tau}(\R^d;- \mu)$ if
\begin{equation}
	\tau > \frac{d}{2} \text{ and } \mu> d\left(\frac{1}{\beta_0} - \frac{1}{2} \right) + \frac{d}{2} = \frac{d}{\beta_0},
\end{equation}
and \eqref{eq:weightedsobolev} is proved.
\end{proof}

\section{L\'evy White Noises on Weighted Besov Spaces} \label{sec:weighted}

We investigate here the Besov smoothness of L\'evy white noises over the complete space $\R^d$. The paths of a nontrivial white noise $w$ are never included in $B_{p,q}^\tau(\R^d)$, since their is no decay at infinity. For this reason, and  as for Sobolev spaces, we consider the weighted Besov spaces, introduced in Section \ref{subsec:besov}. The main course of this section is to prove Theorem \ref{theo:weighted}.
.
\begin{theorem}\label{theo:weighted}
Consider a L\'evy white noise $w$  with Blumenthal-Getoor indices $\beta_0>0 $ and $\beta_\infty$. 
Let $0<p,q \leq \infty$, $\tau, \mu \in \R$.
If  
\begin{equation} \label{eq:weightedbesovregu}
	  \mu> \frac{d}{\min(p,\beta_0)} \text{ and }  \tau > d\left(1 -  \frac{1}{ \max( p, \beta_\infty)}   \right),
\end{equation}
then $w \in B_{p,q}^{- \tau}(\R^d; - \mu)$ a.s.
\end{theorem}
 
\begin{proof}
We start with some preliminary remarks. 
\begin{itemize}
\item First of all, it is sufficient to prove \eqref{eq:weightedbesovregu} for $p=q$, the other cases being deduced by the embedding relations $B_{p,p}^{- \tau + \epsilon}(\R^d;- \mu) \subseteq B_{p,q}^{- \tau}(\R^d;- \mu)$ already seen in \eqref{eq:qdominated}. Therefore, a different parameter $q$ can always be absorbed at the cost of an arbitrarily small smoothness, which is still possible in our case since the condition on $\tau$ in  \eqref{eq:weightedbesovregu} is a strict inequality.  For the same reason, it is admissible to consider that $p< \infty$.

\item Second, we know from Proposition \ref{prop:sobolevregularity} that, for a fixed $\epsilon >0$ and with probability $1$, 
\begin{equation}
w  \in W_2^{-d/2-\epsilon} \left(\R^d;- \frac{d}{\beta_0} -  \epsilon\right).
\end{equation}
From now, we fix $p=q,\tau,\mu$. We can apply Proposition \ref{prop:sobotobesov} with $\tau_0 = -d/2 - \epsilon$ and $\mu_0 = d(\frac{1}{\beta_0} - \frac{1}{2}) + \epsilon$. We set $u$ according to \eqref{eq:conditionu} and consider $(\psi_{j,G,\bm{m}})$ a wavelet basis with regularity $u$ thereafter. 
\end{itemize}

\paragraph{First case: $\beta_\infty < \beta_0$ or $\beta_\infty = \beta_0 \in I_0 \bigcap I_\infty$} We fix $p < \beta \in I_\infty \bigcap I_0$. 
We are by assumption in the conditions of Corollary \ref{coro:momentbound}.  In particular, Corollary \ref{coro:extendedcoro1} applies: The random variable  $\langle w, \varphi\rangle$ is well-defined for any $\varphi \in L_{\beta}(\R^d)$. In particular, since $\beta \in (0,2]$, the Daubechies wavelets $\psi_{j,G,\bm{m}}$, which are compactly supported and in $L_2(\R^d)$, are in $L_{\beta}(\R^d)$, so that the random variables $\langle w,\psi_{j,G,\bm{m}}\rangle$ are well-defined and   \eqref{eq:boundwavelets} is applicable to them. 
We shall show that $w \in B_{p,p}^{-\tau} (\R^d;- \mu)$ a.e.  if
\begin{equation} \label{eq:mutauinterm}
	\mu > d/p \text{ and } \tau > d (1- 1 / \beta).
\end{equation}
To show that $w \in B_{p,p}^{- \tau}(\R^d; - \mu)$ with probability $1$, it is sufficient to show that 
\begin{equation}
	\mathbb{E}\left[ \lVert w \rVert_{B_{p,p}^{-\tau}(\R^d; -  \mu)}^p \right] = \sum_{j\geq 0} \left( 2^{j( - \tau p - d + dp/2)} \sum_{G \in \mathrm{G}^j, \bm{m}\in \Z^d} \frac{ \mathbb{E}[\lvert \langle w,\psi_{j,G,\bm{m}} \rangle \rvert^p] }{ \langle 2^{-j} \bm{m} \rangle^{\mu p}} \right) < \infty.
\end{equation}
The white noise $w$ being stationary, $\mathbb{E}[\lvert \langle w,\psi_{j,G,\bm{m}} \rangle \rvert^p]$ does not depend on the shift parameter $\bm{m}$. Moreover, using \eqref{eq:boundwavelets} with $a = 2^{-j}$, we have that
\begin{equation}
	\mathbb{E}[\lvert \langle w,\psi_{j,G,\bm{m}} \rangle \rvert^p] \leq C 2^{j pd(1/2 - 1/\beta)} \lVert \psi_G \rVert_\beta^p.
\end{equation}
Hence, we deduce that
\begin{equation}
	\mathbb{E}\left[ \lVert w \rVert_{B_{p,p}^{-\tau}(\R^d;-\mu)}^p \right] \leq C' \sum_{j\geq 0} 2^{j (- \tau p - d + dp - dp/\beta)} \sum_{\bm{m}\in \Z^d} \langle 2^{-j} \bm{m} \rangle^{-\mu p},
\end{equation}
where $C' = C\sum_{G\in \mathrm{G}^0} \lVert \psi_G \rVert_\beta^p$ is a finite constant. The sum
$\sum_{\bm{m}\in \Z^d} \langle 2^{-j} \bm{m} \rangle^{-\mu p}$
is finite if and only if $\mu > d / p$, in which case there exists a constant $C_0 >0$ such that
\begin{equation} \label{eq:weightedinterm1}
\sum_{\bm{m}\in \Z^d} \langle 2^{-j} \bm{m} \rangle^{-\mu p} \underset{j\rightarrow \infty}{\sim} C_0 2^{j d}.
\end{equation}
Indeed, we have the convergence of the Riemann sums
\begin{equation}
	\frac{1}{n^d} \sum_{\bm{m}\in \Z^d} \left\langle \frac{\bm{m}}{n} \right\rangle^{-\mu p} \underset{n \rightarrow \infty}{\longrightarrow} \int_{\R^d} \langle \bm{x} \rangle^{- \mu p} \drm \bm{x}  < \infty
\end{equation}
and \eqref{eq:weightedinterm1} is showed setting $2^j = n$, for $C_0 =  \int_{\R^d} \langle \bm{x} \rangle^{-\mu p} \drm \bm{x}$. Finally, for $\mu > d/p$, the quantity $\mathbb{E}\left[ \lVert w \rVert_{B_{p,p}^{-\tau}(\R^d; - \mu)}^p \right] $ is finite if 
\begin{equation}
	\sum_{j\geq 0} 2^{j (- \tau p + dp - dp/ \beta)} < \infty,
\end{equation}
which happens when
\begin{equation}
	\tau -  d + d/\beta > 0.
\end{equation}
We  have shown that $w \in B_{p,p}^{-\tau}(\R^d;-\mu)$ under the conditions of \eqref{eq:mutauinterm}, as expected. We now split the domain of $p$:
\begin{itemize}
	\item if $p \leq \beta_\infty$, by choosing $\beta$ close enough to $\beta_\infty$ (or equal if $\beta_0 = \beta_\infty$), we obtain that $w \in B_{p,p}^{-\tau}(\R^d;- \mu)$ if $\mu> d/p$ and $\tau > d-d/\beta_\infty$;
	\item if $\beta_\infty < p < \beta_0$, by choosing $\beta$ close enough to $p$, we obtain that $w \in B_{p,p}^{-\tau}(\R^d;- \mu)$ if $\mu> d/p$ and $\tau > d-d/p$.
\end{itemize}
We summarize the situation by  $w \in B_{p,p}^{-\tau}(\R^d;- \mu)$ if $\mu> d/p$ and $\tau > d-d/\max(p,\beta_\infty)$, which corresponds to \eqref{eq:weightedbesovregu} for $p < \beta_0$. Finally, the case $p\geq \beta_0$ is deduced from the result for $p<\beta_0$ (by considering values of $p$ arbitrarily close to $\beta_0$) and   the embedding \eqref{eq:conditiontau}.

\paragraph{Second case: general $(\beta_0,\beta_\infty)$} A white noise $w$ can be decomposed as \begin{equation} w = w_1 + w_2, \end{equation}
where $w_1$ and $w_2$ are independent, $w_1$ is a compound-Poisson white noise,  and $w_2$ is finite-variance. To see that, we  invoke  the L\'evy-It\^o decomposition, see for instance \cite[Chapter 4]{Sato1994levy}. It means in particular that $\beta_0(w_1) = \beta_0>0$ and $\beta_\infty( w_1) = 0$. Therefore, $w_1$ is covered by the first case. Moreover, $\beta_\infty (w_2) = \beta_\infty$ and $\beta_0 (w_2) = 2 \in I_0 (w_2)$. Again, $w_2$ is covered by the first case. Indeed, it is obvious if $\beta_\infty <2$. But if $\beta_\infty = 2$, we have that $\beta_\infty(w_2) = \beta_0(w_2) =2 \in I_0 (w_2) \bigcap I_\infty (w_2)$. Hence, $w$ is the sum of two processes $w_1$ and $w_2$ that are in $B_{p,q}^{-\tau}(\R^d;- \mu)$ under the conditions \eqref{eq:weightedbesovregu}. Besov spaces being linear spaces, these conditions are also sufficient for $w$.

\end{proof}

\section{L\'evy White Noises on Local Besov Spaces} \label{sec:local}

The space of infinitely smooth and compactly supported functions is denoted by $\D(\R^d)$. Its topological dual is $\D'(\R^d)$, the space of generalized functions, not necessarily tempered. 
In the same way that we defined generalized random processes over $\S'(\R^d)$, we can also define generalized random processes over $\D'(\R^d)$. This is actually the original approach of Gelfand and Vilenkin in \cite{GelVil4}. 
As we briefly saw in Section \ref{subsec:GRP}, the class of L\'evy white noises over $\D'(\R^d)$ is strictly larger than the one over $\S'(\R^d)$. A L\'evy white noise over $\D'(\R^d)$ is also in $\S'(\R^d)$ if and only if its L\'evy exponent satisfies the Schwartz condition \cite{Dalang2015Levy} or, equivalently, if and only if its Blumenthal-Getoor index $\beta_0$ is not $0$. Until now, we have only considered L\'evy white noises for which $\beta_0 \neq 0$. Since we shall now focus on the local Besov smoothness of L\'evy white noises, the two equivalent conditions are now superfluous.

\begin{definition}
	Let $\tau \in \R$ and $0<p,q \leq \infty$. The \emph{local Besov space $B_{p,q}^{\tau, \mathrm{loc}}(\R^d)$} is the collection of functions $f \in \D'(\R^d)$ such that $f\times \varphi \in B_{p,q}^\tau(\R^d)$ for every $\varphi \in \D(\R^d)$.
\end{definition}

The weighted and local Besov regularities are linked according to Proposition \ref{prop:weightedtolocal}.

\begin{proposition} \label{prop:weightedtolocal}
	Let $\tau, \mu \in \R$, $0<p,q \leq \infty$. We have the continuous embedding
	\begin{equation}
		B_{p,q}^\tau(\R^d;\mu) \subseteq B_{p,q}^{\tau,\mathrm{loc}}(\R^d).
	\end{equation}
\end{proposition}
 
The local regularity of L\'evy white noise is directly obtained from the previous results, essentially up to the case of compound-Poisson noise with $\beta_0 = 0$. Before stating the main result of this section, we therefore have to analyze the compound-Poisson case. 

\begin{definition}
	A \emph{compound-Poisson white noise} is a L\'evy white noise with a L\'evy exponent of the form
	\begin{equation} \label{eq:exponentPoisson}
	f(\xi) = \exp \left(   \lambda (\CFjump (\xi) -1 )  \right),
	\end{equation}
	where $\lambda >0$ is called the \emph{Poisson parameter} and $\Pjump$  is a probability law on $\R\backslash\{0\}$ called the \emph{law of jumps}.
\end{definition}

Compound-Poisson random variables are infinitely divisible \cite{Sato1994levy}, so that \eqref{eq:exponentPoisson} defines a valid L\'evy exponent. 

\begin{lemma} \label{lemma:poisson}
	Let $\tau \in \R$ and $0<p,q\leq \infty$.
	Consider a compound-Poisson noise $w$. If 
	\begin{equation} \label{prop:reguPoisson}
		\tau > d\left( 1 - \frac{1}{p}  \right),
	\end{equation}
	then $w \in B_{p,q}^{-\tau,\mathrm{loc}} (\R^d)$. 
\end{lemma}

\begin{proof}
Let $\lambda$ and $\mathcal{P}_{\mathrm{jump}}$ be  the Poisson parameter and the law of jumps of $w$, respectively. 
	The compound-Poisson noise $w$ can be written as  \cite[Theorem 1]{Unser2011stochastic}
	\begin{equation}
		w = \sum_{n\in \N} a_k \delta ( \cdot - \bm{x}_k),
	\end{equation}
	where $(a_k)$ are i.i.d. with law $\mathcal{P}_{\mathrm{jump}}$ and $(\bm{x}_k)$ are such that the number of $\bm{x}_k$ in any Borelian $B \in \R^d$ is a Poisson random variable with parameter $\lambda \mu(B)$, $\mu$ denoting the Lebesgue measure on $\R^d$. For $\varphi \in \D(\R^d)$, the function $\varphi$ being compactly supported, the generalized random process $w \times \varphi = \sum_{n\in \N} a_n \varphi(\bm{x}_n) \delta( \cdot - \bm{x}_n)$ is almost surely a finite sum of shifted Dirac functions. Hence, it has the Besov regularity of a single Dirac function, which is precisely \eqref{prop:reguPoisson}; see \cite[p. 164]{schmeisser87}. 
\end{proof}

\begin{corollary} \label{coro:local}
Let $0<p,q \leq \infty$, $\tau  \in \R$.
Consider a L\'evy white noise $w$  with indices $\beta_0,\beta_\infty$. 
If 
\begin{equation} \label{eq:localbesovregu}
 	\tau > d\left( \frac{1}{ \max( p, \beta_\infty)} - 1 \right),
\end{equation}
then $w \in B_{p,q}^{-\tau,\mathrm{loc}}(\R^d)$ a.s. 
\end{corollary}

\begin{proof}
	The case $\beta_0 > 0$ is a direct consequence of Theorem \ref{theo:weighted} and Proposition \ref{prop:weightedtolocal}.
	Let assume now that $\beta_0 = 0$. Again, we can split $w$ as $w_1 + w_2$, where $w_1$ is compound-Poisson and $w_2$ is finite-variance. For $w_2$, we can still apply Theorem \ref{theo:weighted}. We can therefore restrict our attention to the case of compound-Poisson noises with $\beta_0 = 0$. But we have seen that the compound-Poisson case---regardless of $\beta_0$--- was covered in Lemma \ref{lemma:poisson}. Since $\beta_\infty = 0$ for compound-Poisson noises, Lemma \ref{lemma:poisson} is consistent with \eqref{eq:localbesovregu}, finishing the proof. 
\end{proof}

\section{Discussion and Examples} \label{sec:examples}

	\subsection{Discussion and Comparison with Known Results}
	
\paragraph{Sobolev regularity of L\'evy white noises}
It is noteworthy to observe that Sections \ref{sec:sobolev} and \ref{sec:weighted}, while based on very different techniques, give  exactly the same results when applied to Sobolev spaces. 
Indeed, applying Theorem \ref{theo:weighted} with $p=q=2$, we recover exactly \eqref{eq:weightedsobolev} due to the relations $\min(2,\beta_0) = \beta_0$ and $\max(2,\beta_\infty) = 2$. Theorem \ref{theo:weighted} is therefore the generalization of Proposition \ref{prop:sobolevregularity}, from Sobolev to Besov spaces. 

Interestingly, the Sobolev smoothness parameter $\tau$ of a L\'evy white noise does not depend on the noise: The universal  sufficient condition is $\tau > d/2$. Moreover, we conjecture that this condition is also necessary, in the sense that $w \notin W_2^{\tau}(\R^d;\mu)$ with probability $1$ for $\tau \geq d/2$ for any $\mu$ and any white noise $w$. 
The situation is   different when considering Besov smoothness for $p\neq 2$.   

\paragraph{H\"older regularity of L\'evy white noises}
We obtain the H\"older regularity of a white noise by setting $p=q = \infty$ in Theorem \ref{theo:weighted}. Because $\min(\infty, \beta_0) = \beta_0$ and $\max(\infty, \beta_\infty) = \infty$, we deduce Corollary \ref{coro:holderregu}.
\begin{corollary} \label{coro:holderregu}
The L\'evy white noise $w$ with Blumenthal-Getoor indices $\beta_0 >0$ and $\beta_\infty$ is in the weighted H\"older space $H^{-\tau}(\R^d;-\mu)$  if
\begin{equation}
	\mu > d/ \beta_0, \quad \tau > d.
\end{equation}
\end{corollary}
As for the Sobolev regularity, the H\"older regularity of a L\'evy noise  that we obtained is independent of the noise. However, the Gaussian white noise has a local H\"older regularity of $(-\tau)$ for every $\tau > \frac{d}{2}$ \cite{Veraar2010regularity}. It means that our bounds for the regularity are suboptimal for the Gaussian white noise. By contrast with the Gaussian case, we conjecture that the condition $\tau > d$ is optimal for non-Gaussian L\'evy white noises. 

\paragraph{The regularity of L\'evy white noises for general $p$}
Fixing the parameters $p=q>0$, we define
\begin{equation}
\tau_p(w) = \min\{\tau \in \R, w \in B_{p,p}^{-\tau,\mathrm{loc}}(\R^d) \ \mbox{a.s.} \}.
\end{equation}
The quantity $\tau_p(w)$ measures the regularity of the L\'evy white noise $w$ for the $L_p$-(quasi-)norm. 
In Corollary   \ref{coro:local}, we have seen that $\tau_p(w) \leq d\left( \frac{1}{ \max( p, \beta_\infty)} - 1 \right)$, a quantity that does not depend on $\beta_0$. We conjecture that 
\begin{equation} \label{eq:taupw}
\tau_p(w) = d\left( \frac{1}{ \max( p, \beta_\infty)} - 1 \right)
\end{equation}
for non-Gaussian white noises. If this is true, then the quantity $\max(p,\beta_\infty)$ is a measure of the regularity of a L\'evy white noise for the $L_p$-(quasi-)norm. 

We summarize the local results of Corollary \ref{coro:local} with the  diagram of Figure \ref{big:plotbesovgeneralcase}. We use the classical $(1/p,\tau)$-representation, which is most convenient for visualization. We indeed see in  \eqref{eq:taupw} that the parameters $1/p$ and $\tau$ are linked with a linear relation for $p\leq \beta_\infty$. 

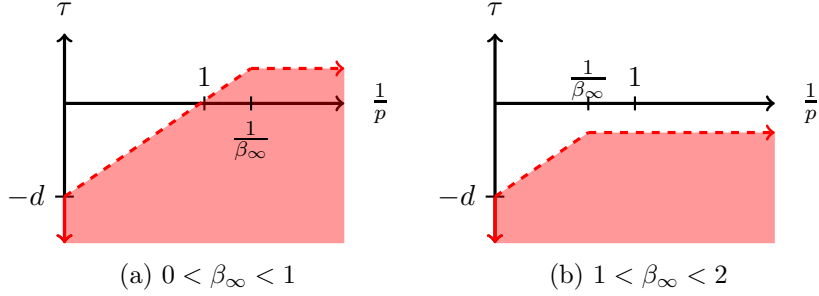
\begin{figure}[t!] 
\centering
\begin{subfigure}[b]{0.35\textwidth}
\begin{tikzpicture}[x=4cm,y=3cm,scale=0.62]

\draw[very thick, ->] (-1,0)--(0.5,0) node[circle,right] {$\frac{1}{p}$} ;
\draw[very thick, <->] (-1,-1)--(-1,0.5) node[circle,above] {$\tau$} ;

\draw[ thick,color=black]
(-1+0.05,-2/3) -- (-1-0.05,-2/3)  node[black,left] { $-d$};

\draw[ thick,color=black]
(-0.25,-0.05) -- (-0.25,0.05)  node[black,above] { $1$};

\fill[red, opacity= 0.4] (-1,-2/3) -- (0,0.25) -- (0.5,0.25) -- (0.5,-1) -- (-1,-1) -- cycle; 
\draw[red, very thick, dashed,->](0,0.25) --(0.5,0.25);
\draw[red, very thick, dashed,- ](-1,-2/3) -- (0,0.25);
\draw[red, very thick,->](-1,-2/3) -- (-1,-1);

\draw[ thick,color=black]
(0,0.05) -- (0,-0.05)  node[black,below] {$\frac{1}{\beta_\infty}$};

\end{tikzpicture}
\caption{$0< \beta_\infty< 1$}
\end{subfigure}
\begin{subfigure}[b]{0.35\textwidth}
\begin{tikzpicture}[x=4cm,y=3cm,scale=0.62]

\draw[very thick, ->] (-1,0)--(0.5,0) node[circle,right] {$\frac{1}{p}$} ;
\draw[very thick, <->] (-1,-1)--(-1,0.5) node[circle,above] {$\tau$} ;

\draw[ thick,color=black]
(-1+0.05,-2/3) -- (-1-0.05,-2/3)  node[black,left] { $-d$};

\draw[ thick,color=black]
(-0.25,-0.05) -- (-0.25,0.05)  node[black,above] { $1$};

\fill[red, opacity= 0.4] (-1,-2/3) -- (-0.5,-5/24) -- (0.5,-5/24) -- (0.5,-1) -- (-1,-1) -- cycle; 
\draw[red, very thick, dashed,->](-0.5,-5/24) -- (0.5,-5/24);
\draw[red, very thick, dashed,- ](-1,-2/3) -- (-0.5,-5/24);
\draw[red, very thick,->](-1,-2/3) -- (-1,-1);

\draw[ thick,color=black]
(-0.5,0.05) -- (-0.5,-0.05)  node[black,above] {$\frac{1}{\beta_\infty}$};

\end{tikzpicture}
\caption{$1< \beta_\infty< 2$}
\end{subfigure}

\caption{Besov localization of general L\'evy white noises. A white noise is almost surely in a given local Besov space $B_{p,q}^{\tau,\mathrm{loc}}(\R^d)$ if $(1/p,\tau)$ is  located in the shaded region.} 
\label{big:plotbesovgeneralcase}
\end{figure}

\paragraph{Weights and L\'evy white noises}
As for the regularity, we can define for $p=q>0$ the optimal weight
\begin{equation}
	\mu_p(w) = \min\{\mu \in \R, \ \exists \tau \in \R, \ w \in B_{p,p}^{-\tau}(\R^d;\mu) \ \mbox{a.s.}  \}.
\end{equation}
According to Theorem \ref{theo:weighted}, we have that $\mu_p(w) \leq \frac{d}{\min(p,\beta_0)}$. We conjecture that 
\begin{equation}
	\mu_p(w) = \frac{d}{\min(p,\beta_0)} 
\end{equation}
for every white noise $w$. 

If this conjecture is true, then $\mu_\infty(w) = 1/\beta_0$. When $\beta_0$ goes to $0$, we need stronger and stronger weights to include the L\'evy white noise into the corresponding H\"older space. The limit case is $\beta_ 0 = 0$ for which we only have local results. Indeed, polynomial weights are not increasing fast enough to compensate the erratic behavior of the white noise. This is consistent with the fact that a L\'evy white noise with $\beta_0 = 0$ is not tempered \cite{Dalang2015Levy}.

	\subsection{Besov Regularity of Some Specific L\'evy White Noises}

We apply our results to specific L\'evy white noises. We start by recalling the Blumenthal-Getoor indices of the considered white noises. We gives in Table \ref{table:noises} the L\'evy exponent and the probability density of the underlying infinitely divisible law, when they can be expressed in a closed form. All the considered distributions are known to be infinitely divisible. For Gaussian, S$\alpha$S, or compound-Poisson noises, this can be easily seen from the definition. For the others, it is a non-trivial fact, and we refer to \cite{Sato1994levy} for more details and references to the adequate literature. 

\begin{table} [t!]
\footnotesize  
\centering
\caption{Blumental-Getoor indices of L\'evy white noises}
\begin{tabular}{lcccccc} 
\hline
\hline 
White noise & parameter & $f(\xi)$ & $p_{\mathrm{id}} (x) $ & $\beta_0$ & $\beta_\infty$ & \emph{cf.}\\
\hline\\[-1ex]
Gaussian & $\sigma^2 >0$ & $- \sigma^2 \xi^2 / 2$ & $\frac{\mathrm{e}^{- x^2 / 2\sigma^2}}{\sqrt{2\pi\sigma^2}} $ & $2$ & $2$ & \cite{Veraar2010regularity} \\
Pure drift & $\mu \in \R$ & $\mathrm{i} \mu \xi$ & $\delta( \cdot - \mu)$ & $1$ & $1$ &   \\
S$\alpha$S  & $\alpha \in (0,2)$  & $-\lvert \xi \rvert^\alpha$ & --- & $\alpha$ & $\alpha$ & \cite{Taqqu1994stable}  \\
Sum of S$\alpha$S & $\alpha,\beta \in (0,2)$ & $-\lvert \xi \rvert^\alpha -\lvert \xi \rvert^\beta$ & --- & $\min(\alpha,\beta)$ & $\max(\alpha,\beta)$ & \cite{Taqqu1994stable}  \\
Laplace & --- & $- \log ( 1 + \xi^2) $ & $\frac{1}{2} \mathrm{e}^{-\lvert x\rvert}$ & $2$ & $0$ & \cite{Koltz2001laplace} \\
Sym-gamma & $c>0$ & $- c \log(1 + \xi^2)$ & ---  & $2$ & $0$ & \cite{Koltz2001laplace} \\
Poisson & $\lambda >0$ & $\lambda (\mathrm{e}^{\mathrm{i} \xi} -1)$ & --- & $2$ & $0$ &   \\
compound-Poisson  & $\lambda > 0, \mathbb{P}_J$ & $\lambda (\widehat{\mathbb{P}}_J (\xi) - 1)$ &--- & variable   & $0$ & \cite{Unser2014sparse} \\
Inverse Gaussian & --- & --- & $\frac{ \mathrm{e}^{-x}}{\sqrt{2\pi} x^{3/2}} \One_{x\geq 0}$ & $2$ & $1/2$ & \cite{Barndorff1997processes} \\ 
\hline
\hline
\end{tabular} \label{table:noises}
\end{table}

We moreover remark that any combination of $\beta_0$ and $\beta_\infty$ is possible, as stated in Proposition \ref{prop:allcases}. 

\begin{proposition} \label{prop:allcases}
	For every $\beta_0, \beta_\infty \in [0,2]$, there exists a L\'evy white noise with Blumenthal-Getoor indices $\beta_0$ and $\beta_\infty$. 
\end{proposition}

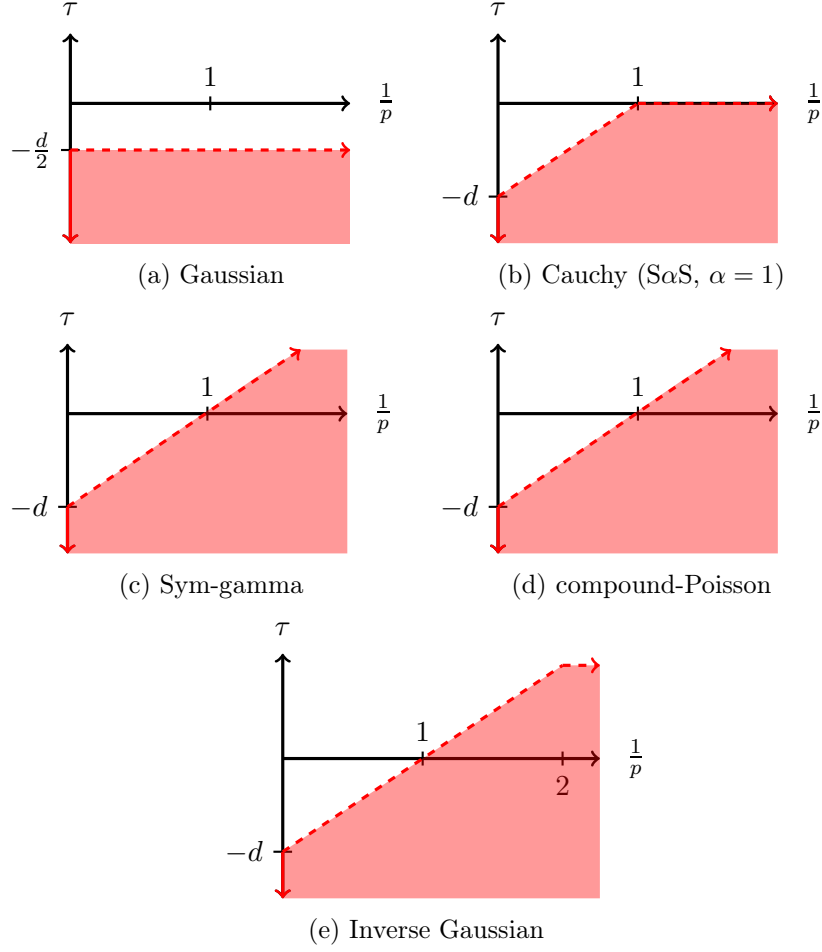
\begin{figure}[t!] 
\centering
\begin{subfigure}[b]{0.35\textwidth}
\begin{tikzpicture}[x=4cm,y=3cm,scale=0.62]

\draw[very thick, ->] (-1,0)--(0.5,0) node[circle,right] {$\frac{1}{p}$} ;
\draw[very thick, <->] (-1,-1)--(-1,0.5) node[circle,above] {$\tau$} ;

\draw[ thick,color=black]
(-1+0.05,-1/3) -- (-1-0.05,-1/3)  node[black,left] { $-\frac{d}{2}$};

\fill[red, opacity= 0.4] (-1,-1) -- (-1,-1/3) -- (0.5,-1/3) -- (0.5,-1) -- cycle; 
\draw[red, very thick,-> ](-1,-1/3) -- (-1,-1);
\draw[red, very thick, dashed,->](-1,-1/3) --(0.5,-1/3);

\draw[ thick,color=black]
(-0.25,-0.05) -- (-0.25,0.05)  node[black,above] { $1$};
\end{tikzpicture}
\caption{Gaussian}
\end{subfigure}
\begin{subfigure}[b]{0.35\textwidth}
\begin{tikzpicture}[x=4cm,y=3cm,scale=0.62]

\draw[very thick, ->] (-1,0)--(0.5,0) node[circle,right] {$\frac{1}{p}$} ;
\draw[very thick, <->] (-1,-1)--(-1,0.5) node[circle,above] {$\tau$} ;

\draw[ thick,color=black]
(-1+0.05,-2/3) -- (-1-0.05,-2/3)  node[black,left] { $-d$};

\draw[ thick,color=black]
(-0.25,-0.05) -- (-0.25,0.05)  node[black,above] { $1$};

\fill[red, opacity= 0.4] (-1,-2/3) -- (-0.25,0) -- (0.5,0) -- (0.5,-1) -- (-1,-1) -- cycle; 
\draw[red, very thick, dashed,->](-0.25,0) -- (0.5,0) ;
\draw[red, very thick, dashed,- ](-1,-2/3) --(-0.25,0);
\draw[red, very thick,->](-1,-2/3) -- (-1,-1);

\end{tikzpicture}
\caption{Cauchy (S$\alpha$S, $\alpha = 1$)}
\end{subfigure}
\begin{subfigure}[b]{0.35\textwidth}
\begin{tikzpicture}[x=4cm,y=3cm,scale=0.62]

\draw[very thick, ->] (-1,0)--(0.5,0) node[circle,right] {$\frac{1}{p}$} ;
\draw[very thick, <->] (-1,-1)--(-1,0.5) node[circle,above] {$\tau$} ;

\draw[ thick,color=black]
(-1+0.05,-2/3) -- (-1-0.05,-2/3)  node[black,left] { $-d$};

\draw[ thick,color=black]
(-0.25,-0.05) -- (-0.25,0.05)  node[black,above] { $1$};

\fill[red, opacity= 0.4] (-1,-2/3) -- (0,0.25) -- (0.25,11/24) -- (0.5,11/24) -- (0.5,-1) -- (-1,-1) -- cycle; 
\draw[red, very thick, dashed,-> ](-1,-2/3) -- (0.25,11/24);
\draw[red, very thick,->](-1,-2/3) -- (-1,-1);

\end{tikzpicture}
\caption{Sym-gamma}
\end{subfigure}
\begin{subfigure}[b]{0.35\textwidth}
\begin{tikzpicture}[x=4cm,y=3cm,scale=0.62]

\draw[very thick, ->] (-1,0)--(0.5,0) node[circle,right] {$\frac{1}{p}$} ;
\draw[very thick, <->] (-1,-1)--(-1,0.5) node[circle,above] {$\tau$} ;

\draw[ thick,color=black]
(-1+0.05,-2/3) -- (-1-0.05,-2/3)  node[black,left] { $-d$};

\draw[ thick,color=black]
(-0.25,-0.05) -- (-0.25,0.05)  node[black,above] { $1$};

\fill[red, opacity= 0.4] (-1,-2/3) -- (0,0.25) -- (0.25,11/24) -- (0.5,11/24) -- (0.5,-1) -- (-1,-1) -- cycle; 
\draw[red, very thick, dashed,-> ](-1,-2/3) -- (0.25,11/24);
\draw[red, very thick,->](-1,-2/3) -- (-1,-1);

\end{tikzpicture}
\caption{compound-Poisson}
\end{subfigure}
\begin{subfigure}[b]{0.35\textwidth}
\begin{tikzpicture}[x=4cm,y=3cm,scale=0.62]

\draw[very thick, ->] (-1,0)--(0.7,0) node[circle,right] {$\frac{1}{p}$} ;
\draw[very thick, <->] (-1,-1)--(-1,0.75) node[circle,above] {$\tau$} ;

\draw[ thick,color=black]
(-1+0.05,-2/3) -- (-1-0.05,-2/3)  node[black,left] { $-d$};

\draw[ thick,color=black]
(-0.25,-0.05) -- (-0.25,0.05)  node[black,above] { $1$};

\draw[ thick,color=black]
(0.5,0.05) -- (0.5,-0.05)  node[black,below] {$2$};

\fill[red, opacity= 0.4] (-1,-2/3) --  (0.5,2/3) -- (0.7,2/3) -- (0.7,-1) -- (-1,-1) -- cycle; 
\draw[red, very thick, dashed,- ](-1,-2/3) --  (0.5,2/3);
\draw[red, very thick, dashed,->](0.5,2/3) --  (0.7,2/3);
\draw[red, very thick,->](-1,-2/3) -- (-1,-1);

\end{tikzpicture}
\caption{Inverse Gaussian}
\end{subfigure}

\caption{Besov localization of specific L\'evy white noises. A white noise is almost surely in a given Besov space $B_{p,q}^{\tau,\mathrm{loc}}(\R^d)$ if $(1/p,\tau)$ is  located in the shaded region.} 
\label{fig:plotbesov}
\end{figure}

\begin{proof}
We first recall that a L\'evy exponent can be uniquely represents by its L\'evy triplet $(\mu,\sigma^2,\nu)$ as \cite[Theorem 8.1]{Sato1994levy}
\begin{equation}
	f(\xi) = \mathrm{i} \mu \xi - \frac{\sigma^2 \xi^2}{2} + \int_{\Rstar} (\mathrm{e}^{\mathrm{i} \xi x} - 1 - \mathrm{i}\xi x \One_{\lvert x \rvert \leq 1}) \nu(\drm x),
\end{equation}
where $\mu \in \R$, $\sigma^2 \geq 0$, and $\nu$ is a L\'evy measure, that is, a measure on $\Rstar$ such that $\int_{\Rstar} \inf(1,x^2)  \nu(\drm x)< \infty$. When $\mu = \sigma^2 = 0$ and $\nu$ is symmetric, the L\'evy white noise with L\'evy triplet $(0,0,\nu)$ has Blumenthal-Getoor indices given by \cite[Section 3.1]{Deng2015shift}
\begin{align} \label{eq:indiceslevymeasure}
	\beta_\infty &= \inf_{p \in [0,2]}  \left\{ \int_{\lvert x \rvert \leq 1} \lvert x \rvert^p \nu(\drm x)< \infty \right\} , \nonumber \\
	\beta_0 &= \sup_{p \in [0,2]}  \left\{ \int_{\lvert x \rvert > 1} \lvert x \rvert^p \nu(\drm x) < \infty  \right\}.
\end{align}
For $0<\beta_0 \leq 2$ and $0 \leq \beta_\infty<2$, we set $$\nu_{\beta_0} (x) =   \lvert x \rvert^{-(\beta_0 +1)} \One_{\lvert x \rvert>1}  , \quad \nu^{\beta_\infty}(  x) = \ \lvert x \rvert^{-(\beta_\infty +1)} \One_{\lvert x \rvert \leq 1} .$$ 
Moreover, for $\beta_0 = 0$ and $\beta_\infty = 2$, we set $$\nu_{0} (x) = \log^{-2} (1 + \lvert x \rvert ) \lvert x \rvert^{-1} \One_{\lvert x \rvert>1}, \quad \nu^2 (x) = \log^{-1/2} (1 + \lvert x \rvert)  \lvert x \rvert^{-3}  \One_{\lvert x \rvert \leq 1}.$$
For $0\leq \beta_0 , \beta_ \infty \leq 2$ and defining $\nu_{\beta_0}^{\beta_\infty} = \nu_{\beta_0}  + \nu^{\beta_\infty}$, we see easily that $$\int_{\Rstar} \inf(1,x^2) \nu_{\beta_0}^{\beta_\infty}(\drm x)  < \infty,$$ so that $\nu_{\beta_0}^{\beta_\infty}$ is a L\'evy measure. Based on  \eqref{eq:indiceslevymeasure}, we also see  that  the associated indices are $\beta_0$ and $\beta_\infty$.
\end{proof}

	\subsection{Conclusion}
We derived new results on the Besov regularity of general multidimensional L\'evy white noises. We have shown that the local regularity depends on the Blumenthal-Getoor index $\beta_\infty$ which measures the behavior of the L\'evy exponent at infinity. The global regularity over the complete space $\R^d$ requires the introduction of polynomial weights. We have shown that the second Blumenthal-Getoor index $\beta_0$ characterizes the weight parameter. Finally, we applied our results to specific white noises. 

\section*{Acknowledgment}
The authors are grateful to John Paul Ward for the fruitful preliminary discussions that lead to this work. The research leading to these results has received funding from the European Research Council under the European Union's Seventh Framework Programme (FP7/2007-2013) / ERC grant agreement $\text{n}^\circ$ 267439.

\bibliographystyle{plain}
\bibliography{references}

\end{document}